\documentclass[11pt,reqno]{amsart}
\usepackage{graphicx}
\usepackage{enumerate}
\usepackage{amssymb}

\usepackage[dvips]{psfrag}

\newcommand{\eqdef}{\stackrel{\scriptscriptstyle\rm def}{=}}

\def\loc{{\operatorname{loc}}}

\def\dist{{\operatorname{dist}}}

\def\ZZ {{\mathbb Z}}
\def\NN {{\mathbb N}}

\def\RR {{\mathbb R}}
\def\CC {{\mathbb C}}

\def\Si{\Sigma}
\def\La{\Lambda}
\def\De{\Delta}

\def\la{\lambda}

\def\de{\delta}

\def\ve{\varepsilon}
\def\eps{\varepsilon}
\def\vro{\varrho}

\def\cA{{\mathcal A}}  \def\cG{{\mathcal G}}  
\def\cB{{\mathcal B}}   \def\cN{{\mathcal N}} 
\def\cC{{\mathcal C}}    \def\cU{{\mathcal U}}
   \def\cP{{\mathcal P}} \def\cV{{\mathcal V}}
    \def\cW{{\mathcal W}}
   \def\cR{{\mathcal R}} 
  
\def\fP{{\mathfrak{P}}}
\def\fQ{{\mathfrak{P}}}
\def\fV{{\mathfrak{V}}}

\def\fC{{\mathfrak{U}}}

\def\diff{{\mbox{{\rm Diff}}^1(M)}}

\def\diffr{{\mbox{{\rm Diff}}^r(M)}}

\def\pes{\varepsilon}

\newtheorem{theo}{Theorem}
\newtheorem{coro}[theo]{Corollary}
\newtheorem{theor}{Theorem}[section]
\newtheorem{lemm}[theor]{Lemma}
\newtheorem{clai}[theor]{Claim}

\newtheorem{corol}[theor]{Corollary}

\newtheorem{prop}[theor]{Proposition}

\newtheorem{conj}{Conjecture}
\newtheorem{ques}{Question}
\newtheorem{rema}[theor]{Remark}
\newtheorem{defi}[theor]{Definition}

\begin{document}


\title[Non-domination, cycles, and self-replication]{Internal perturbations of homoclinic classes:\\
non-domination, cycles, and self-replication}

\author{Ch. Bonatti}
\address{Institut de Math\'ematiques de Bourgogne, BP 47 870, 1078 Dijon Cedex, France}
\email{bonatti@u-bourgogne.fr}

\author{S. Crovisier}
\address{Institut Galil\'ee, Universit\'e Paris 13, Avenue J.-B. Cl\'ement, 93430 Villetaneuse, France}
\email{crovisie@math.univ-paris13.fr}

\author{L. J. D\'iaz}
\address{Depto. Matem\'atica, PUC-Rio, Marqu\^es de S. Vicente
225, 22453-900 Rio de Janeiro RJ,  Brazil}
\email{lodiaz@mat.puc-rio.br}

\author{N. Gourmelon}
\address{Institut de Math\'ematiques de Bordeaux, Universit\'e de Bordeaux, 351, cours de la Lib\'eration, F33405 Talance Cedex France}
\email{nicholas.gourmelon@math.u-bordeaux1.fr}

\thanks{This paper was partially supported by CNPq, Faperj, Pronex
(Brazil), the ANR project DynNonHyp BLAN08-2$_-$313375,  and the
Agreement Brazil-France in Mathematics. LJD thanks the kind
hospitality of LAGA at Univ. Paris 13 and the Institut de
Math\'ematiques of Universit\'e de Bourgogne. SC and LJD thanks
the financial support and the warm hospitality of Institute
Mittag-Leffer. We are grateful for comments by K. Shinohara that allowed us to improve the
presentation of Section~\ref{s.viral}.}


\date{\today}


\begin{abstract}
Conditions are provided under which lack of domination of a
homoclinic class yields robust heterodimensional cycles. Moreover,
so-called viral homoclinic classes are studied. Viral classes have
the property of generating copies of themselves  producing wild
dynamics (systems with infinitely many homoclinic classes with
some persistence). Such wild dynamics also exhibits uncountably
many aperiodic chain recurrence classes. A scenario (related with
non-dominated dynamics) is presented where viral homoclinic
classes occur.

A key ingredient are adapted perturbations of  a diffeomorphism
along a periodic  orbit. Such perturbations preserve certain
homoclinic relations and prescribed dynamical properties of a
homoclinic class.


\end{abstract}

\maketitle


\section{Introduction}\label{s.introd}

There are two sort of cycles associated to periodic saddles that
are the main mechanism for breaking hyperbolicity of  systems:

\smallskip

\noindent $\bullet$
 {\bf Homoclinic tangencies:} A diffeomorphism $f$ has a
  {\emph{homoclinic tangency}} associated to a
 transitive hyperbolic set $K$ if there are points $X$ and $Y$ in $K$ whose stable and unstable manifolds have some
non-transverse intersection. The homoclinic tangency is
{\emph{$C^r$-robust}} if there is a $C^r$-neighborhood $\cN$ of
$f$ such that the hyperbolic continuation $K_g$ of $K$ has a
homoclinic tangency for every $g\in \cN$.

\smallskip

\noindent  $\bullet$ {\bf{Heterodimensional cycles:}} A
diffeomorphism $f$ has a \emph{heterodimensional cycle} associated
to a pair of transitive hyperbolic sets $K$ and $L$ of $f$ if
their stable bundles have different dimensions and their invariant
manifolds meet cyclically, that is, $W^s(K) \cap W^u (L)\ne
\emptyset$ and $W^u(K) \cap W^s (L)\ne \emptyset$. The
heterodimensional cycle is {\emph{$C^r$-robust}} if there is a
$C^r$-neighborhood $\cV$ of $f$ such that the continuations $K_g$
and $L_g$ of $K$ and $L$ have a heterodimensional cycle for every
$g\in \cV$.

Given a closed manifold $M$ consider the space $\diffr$ of
$C^r$-diffeo\-mor\-phisms defined on $M$ endowed with the usual
$C^r$-topology. There is the following conjecture about
hyperbolicity and cycles:

\begin{conj}[Palis' density conjecture, \cite{P:00}]
Any diffeomorphism $f\in \diffr$, $r\ge 1$, can be
$C^r$-approximated either by a hyperbolic diffeomorphism (i.e.
satisfying the Axiom A and the no-cycles condition) or by a
diffeomorphism that exhibits a homoclinic tangency or a
heterodimensional cycle.
\end{conj}
This conjecture was proved for $C^1$-surface diffeomorphisms  in
\cite{PS:00}. For some partial progress in higher dimensions see
\cite{BD:08,CP:prep}.

Besides this conjecture one also aims to understand the dynamical
phenomena associated to homoclinic tangencies and
heterodimensional cycles and the interplay between them. We
discuss these topics in the next paragraphs.

\medskip

Homoclinic tangencies of $C^2$-diffeomorphisms are the main source
of non-hyperbolic dynamics in dimension two, see
\cite{PT:93,N:04}. Namely, as a key mechanism a homoclinic
tangency of a surface $C^2$-diffeomorphism yields $C^2$-robust
homoclinic tangencies and  generates open sets of diffeomorphisms
where the generic systems display infinitely many sinks or
sources, \cite{N:78,N:79}. This  leads to the first examples of
the so-called {\emph{wild dynamics}} (i.e. systems having
infinitely many elementary pieces of dynamics with some
persistence, see \cite[Chapter 10]{BDV:04} for a discussion and
precise definitions). Moreover, these homoclinic tangencies also
yield infinitely many regions containing robust homoclinic
tangencies associated to other hyperbolic sets (this follows from
 \cite{N:79} and \cite{Co:98}, see also the comments in
\cite[page 33]{BDV:04}). Using the terminology in \cite{B:bible},
this means that, for surface diffeomorphisms, the existence of
$C^2$-robust tangencies  is a \emph{self-replicating} or
\emph{viral} property, for more details see
Section~\ref{ss.viral}.

Comparing with the $C^2$-case, $C^1$-diffeomorphisms of surfaces
do not have hyperbolic sets with robust homoclinic tangencies, see
\cite{M:pre} and also \cite[Corollary 3.5]{B:bible} for a formal
statement. However, in higher dimensions $C^1$-diffeomorphisms can
display robust tangencies, see for instance \cite{S:72,A:08}.

In higher dimensions, the first examples of robustly
non-hyperbolic dynamics were obtained by Abraham and Smale in
\cite{AS:70} by constructing diffeomorphisms with robust
heterodimensional cycles (although this terminology is not used there).
Moreover, the diffeomorphisms with heterodimensional cycles in
\cite{AS:70} also exhibit robust homoclinic tangencies (this
follows from \cite{BD:pre}).

In the $C^1$-setting, the generation of homoclinic tangencies is a
quite well understood phenomenon that is strongly related to the
existence of non-dominated splittings, \cite{W:04,BGV:06,G:10}.
Contrary to the case of tangencies, the generation of
heterodimensional cycles is not well understood and remains the
main difficulty for solving Palis conjecture in the $C^1$-case. In
contrast with the case of $C^1$-homoclinic tangencies,
heterodimensional cycles yield $C^1$-robust cycles after small
$C^1$-perturbations, \cite{BD:08}. However, in dimension $d\ge 3$,
 we do not know ``when and how" homoclinic tangencies
may occur $C^1$-robustly. In fact, all known examples of
$C^1$-robust tangencies also exhibit $C^1$-robust
heterodimensional cycles\footnote{The converse is false: there are
diffeomorphisms (of partially hyperbolic type with one dimensional
central direction) that display robust heterodimensional cycles
but cannot have homoclinic tangencies, see  for instance
\cite{M:78,BD:95}.}. For further discussion see \cite[Conjecture
6]{B:bible}.

 These comments lead to the following
strong version of Palis' conjecture (in fact, this reformulates
\cite[Question 1]{BD:08}):

\begin{conj}[{\cite[Conjecture 7]{B:bible}}]
\label{c.bonatti} The union of the set of hyperbolic
diffeomorphisms (i.e. satisfying the Axiom A and the no-cycle
condition) and  of the set of diffeomorphisms having a robust
heterodimensional cycle is
dense in $\diff$. 
\end{conj}

This conjecture holds in two relevant $C^1$-settings: the
conservative diffeomorphisms in dimension $d\ge 3$ and the so
called \emph{tame} systems (diffeomorphisms whose chain recurrence
classes are robustly isolated), see \cite{C:pre} and
 \cite[Theorem 2]{BD:08}. See also
previous results in \cite{A:03,GW:03}.

\subsection{Some informal statements and
questions}\label{ss.informal} In what follows we focus on
$C^1$-diffeomorphisms defined on closed manifolds of dimension
$d\ge 3$. We now briefly and roughly describe some of our results
and the sort of questions we will consider (the precise
definitions and statements will be given throughout the
introduction).

\smallskip

\noindent {\bf {A)}}
 {\em When do homoclinic tangencies yield heterodimensional cycles?}
In terms of {\emph{dominated splittings,}} Theorem~\ref{t.main} and
Corollary~\ref{c.main} give a natural setting where homoclinic
tangencies generate  heterodimensional cycles after arbitrarily
small $C^1$-perturbations.

\smallskip

\noindent {\bf {B)}} {\em What are  obstructions to the occurrence
of heterodimensional cycles?} {\em Sectional dissipativity}
prevents the ``coexistence" of periodic saddles with different
indices and hence the occurrence of heterodimensional cycles. For
homoclinic classes that do not have dominated splittings,  we
wonder  if this is the only possible obstruction for the
generation of heterodimensional cycles.
Corollary~\ref{c.excnodominated} shows that sectional
dissipativity is indeed the only obstruction for the occurrence of
heterodimensional cycles in homoclinic classes without any
dominated splitting.

\smallskip

\noindent {\bf {C)}} {\em Is it possible to turn the lack of
domination into a robust property?} For homoclinic classes,
Theorem~\ref{t.complex} shows that the non-existence of a
{\emph{dominated splitting of index $i$}} can always be made a
robust property when the class contains some saddle of stable
index different from $i$.

\smallskip

\noindent {\bf {D)}} {\em Which are the dynamical features
associated to robust non-do\-mi\-na\-ted dynamics?} In contrast to
the case of surfaces, homoclinic tangencies and ``some" lack of
domination do not always lead to wild dynamics. A homoclinic
tangency corresponds to the lack of domination of some index. For
homoclinic classes containing saddles of several stable indices,
Theorem~\ref{t.bviral} and Corollary~\ref{c.bviral} claim that the
robust lack of any domination leads to wild dynamics. In fact,
Theorem~\ref{t.bviral} asserts that the property of ``total
non-domination plus  coexistence of saddles of several indices"
provides another example of a viral property of a chain recurrence
class. This property leads to the generic coexistence of a
non-countable set of different (aperiodic) classes, extending
previous results in \cite{BD:02}

\smallskip

We next define precisely the main definitions involved in this
paper and state our main results.

\subsection{Basic definitions}\label{ss.basic}

 We will focus on two types of elementary pieces of the dynamics:
   homoclinic classes and chain recurrence classes.

 The {\emph{homoclinic class}} of a hyperbolic periodic point $P$,
denoted by $H(P,f)$, is the closure of the transverse
intersections of the stable and unstable manifolds of the orbit of
$P$. Note that the class $H(P,f)$ coincides with the closure of
the saddles $Q$ {\emph{homoclinically related with $P$}}: the
stable manifold of the orbit of $Q$ transversely meets the
unstable manifold of the orbit of $P$ and vice-versa.

To define a {\emph{chain recurrence class}}  we need some
preparatory definitions. A finite sequence of points
$(X_i)_{i=0}^n$ is an {\emph{$\epsilon$-pseudo-orbit}} of a
diffeomorphism $f$ if $\mbox{dist\,}(f(X_i),X_{i+1})<\epsilon$ for
all $i=0,\dots,n-1$. A point $X$ is {\emph{chain recurrent}} for
$f$ if
 for every $\epsilon>0$ there is an $\epsilon$-pseudo-orbit
 $(X_i)_{i=0}^n$, $n\ge 1$,
starting and ending at $X$ (i.e. $X=X_0=X_n$). The chain recurrent
points form the {\emph{chain recurrent set}} of $f$, denoted by
$R(f)$. This set splits into disjoint {\emph{chain recurrence
classes}} defined as follows. The class of a point $X\in R(f)$,
denoted by $C(X,f)$, is the set of points $Y\in M$ such that for
every $\epsilon>0$ there are $\epsilon$-pseudo-orbits joining $X$
to $Y$ and $Y$ to $X$. A chain recurrence class that does not
contain periodic points is called {\emph{aperiodic.}}

As a remark, in general, for hyperbolic periodic points their
chain recurrence classes contain  their homoclinic ones. However,
for $C^1$-generic diffeomorphisms the equality holds,
\cite[Remarque 1.10]{BC:04}.

A key ingredient in this paper is the notion of {\emph{dominated
splitting}}:

\begin{defi}[Dominated splitting]\label{d.dominated}
Consider a diffeomorphism $f$ and a compact $f$-invariant set
$\La$. A $Df$-invariant splitting $T_{\La}M=E\oplus F$ over $\La$
is {\emph{dominated}} if the fibers $E_x$ and $F_x$ of $E$ and $F$
have constant dimensions and there exists $k\in \NN$ such that
\begin{equation}
\label{e.dominated} \frac{||D_x f^k(u)||}{||D_xf^{k} (w) ||}\le
\frac{1}{2},
\end{equation}
 for
every $x\in \La$ and every pair of unitary vectors $u\in E_x$ and
$w\in F_x$.

The {\emph{index}} of the dominated splitting is the dimension of
$E$.

 When we want to stress on the role  of the constant $k$ we say
that the splitting is \emph{$k$-dominated.}
\end{defi}

Given a periodic point $P$ of $f\in \diff$ denote by $\pi(P)$ its
period. We order the eigenvalues $\la_1(P),\dots, \la_d(P)$ of
$D_P f^{\pi(P)}$ in increasing modulus and counted with
multiplicity, that is, $|\la_i(P)|\le |\la_{i+1}(P)|$. We call
$\la_i(P)$ the {\emph{$i$-th multiplier}} of $P$. The \emph{$i$-th
Lyapunov exponent} of $P$ is $\chi_i(P)=\frac{1}{\pi(P)} \, \log
|\la_i(P)|$. If $\chi_i(P)<\chi_{i+1}(P)<0$ then one can define
the {\emph{strong stable manifold of dimension $i$}} of the orbit
of $P$, denoted by $W^{ss}_i(P,f)$, as the only $f$-invariant
embedded manifold of dimension $i$ tangent to the $i$-dimensional
eigenspace corresponding to the multipliers $\la_1(P),\dots ,
\la_i(P)$. There are similar definitions for strong unstable
manifolds.

Recall that if $\La$ is hyperbolic set of $f$ then every
diffeomorphism $g$ close to $f$ has a hyperbolic set $\La_g$
(called the {\emph{continuation of $\La$}}) that is close and
conjugate to $\La$. If the set $\La$ is transitive the dimension
of its stable bundle is called its {\emph{stable index}} or simply
\emph{$s$-index}.

Throughout this paper we consider diffeomorphisms defined on
closed manifolds of dimension $d\ge 3$. Unless it is explicitly
mentioned, we always consider $C^1$-diffeomorphisms,
$C^1$-neighborhoods, and so on. We repeatedly  consider
perturbations of diffeomorphisms. By a {\emph{perturbation}} of a
diffeomorphism $f$ we mean here a diffeomorphism $g$ that is
arbitrarily $C^1$-close to $f$. To emphasize the size of the
perturbation we say that a diffeomorphism $g$ is a
{\emph{$\pes$-perturbation}} of $f\in \diff$ if the $C^1$-distance
between $f$ and $g$ is less than $\pes$.

\subsection{Heterodimensional cycles generated by homoclinic
tangencies}\label{ss.heterodimtangencies} Recall that the
generation of homoclinic tangencies  is closely related to the
absence of dominated splittings over homoclinic classes. In fact,
in \cite{G:10} it is proved that if the stable/unstable splitting
over the periodic points homoclinically related to a saddle $P$ is
not dominated then there are diffeomorphisms $g$ arbitrarily
$C^1$-close to $f$ with a homoclinic tangency associated to $P_g$.
See also previous results in \cite{W:04}.

Our main result about the interplay between homoclinic tangencies
and heterodimensional cycles is stated in the following theorem.

\begin{theo}\label{t.main} Let $f$ be a diffeomorphism and
$P$ a hyperbolic periodic saddle of $f$ with stable index $i\ge
2$.
Assume that
\begin{enumerate}
\item\label{i1}
there is no dominated splitting over $H(P,f)$ of index $i$,
\item\label{i2}
there is no dominated splitting over $H(P,f)$ of index $i-1$, and
\item\label{i3new} the  Lyapunov exponents of  $P$ satisfy
$\chi_i(P)+\chi_{i+1}(P)\geq 0$.
\end{enumerate}
Then there are diffeomorphisms $g\in \diff$ arbitrarily
$C^1$-close to $f$ with a heterodimensional cycle associated to
$P_g$ and a saddle $R_g\in H(P_g,g)$ of stable index $i-1$.

Moreover, the diffeomorphisms $g$ can be chosen such that there
are hyperbolic transitive sets $L_g$ and $K_g$ containing $P_g$
and $R_g$, respectively, having simultaneously a robust
heterodimensional cycle and
 a robust homoclinic tangency.
\end{theo}

\begin{rema}\label{r.mainb} $\,$

\noindent (i) In fact, we prove  Theorem~\ref{t.main} under the
following slightly weaker hypothesis  replacing condition
(\ref{i3new}).
\begin{enumerate}
\item[(3')]
For every $\de>0$ there exists a periodic point $Q_\de$
homoclinically related to $P$ whose Lyapunov exponents satisfy
$\chi_i(Q_\de)+\chi_{i+1}(Q_\de)\geq -\de$.
\end{enumerate}
\smallskip

\noindent
(ii) Hypothesis (\ref{i2}) can be replaced by the
following condition (see Proposition~\ref{p.excfinal}).
\begin{enumerate}
\item[(2')]  There is a diffeomorphism $g$ arbitrarily $C^1$-close to $f$  having a periodic point $R_g$ that is homoclinically related to $P_g$
and that has a strong stable manifold of dimension $i-1$
intersecting the unstable manifold of the orbit of $R_g$.
\end{enumerate}
\end{rema}

Theorem~\ref{t.main} will be proved in
Section~\ref{ss.proofoftmain}. Let us observe that
for three dimensional
diffeomorphisms
 a version of
this theorem was proved in \cite{S:pre}  replacing condition
(\ref{i3new}) by a stronger one requiring  existence of a saddle
$Q$ homoclinically related to $P$ such that $\chi_1(Q)+\chi_{2}(Q)
+\chi_3(Q)>0$. Note that conditions (\ref{i3new}) and (3') are
related to the notion of a sectionally dissipative bundle that is
also considered in \cite{PV:94,R:95}, see
Section~\ref{ss.farfromheterodimnsional}.

Condition~(\ref{i1}) is used to get homoclinic tangencies
associated to $P$. Conditions~(\ref{i2}) and (\ref{i2}') assure
that the homoclinic class is not contained in a normally
hyperbolic surface (this would be an obstruction for the
generation of heterodimensional cycles). Finally,
condition~(\ref{i3new}) implies that these tangencies generate
saddles of index $i-1$.

We would like to replace condition (\ref{i3new}) (or (3')) by a
weaker one about Lyapunov exponents of measures supported over the
class, namely requiring the existence of an ergodic measure $\mu$
whose $i$-th and $(i+1)$-th Lyapunov exponents satisfy
$\chi_i(\mu)+\chi_{i+1}(\mu)\geq 0$. This potential extension is
related to the still open problem of approximation of ergodic
measures supported on a homoclinic class by measures supported on
periodic points of the class, see \cite[Conjecture 2]{B:bible} and
\cite{ABC:}.

There is also the following ``somewhat symmetric" version of
Theorem~\ref{t.main} that is an immediate consequence of it.

\begin{coro}\label{c.main}
Consider  a hyperbolic  saddle $P$ of stable index $i$, $2\leq
i\leq d-2$, of a diffeomorphism $f$. Assume that there are no
dominated splittings over $H(P,f)$ of indices  $i-1$, $i$, and
$i+1$. Then there is a diffeomorphism $g$ arbitrarily $C^1$-close
to $f$ with a heterodimensional cycle associated to $P_g$ and a
saddle $R_g\in H(P_g,g)$ of stable index $i-1$ or $i+1$.

Moreover, the diffeomorphism $g$ can be chosen such that there are
hyperbolic transitive sets $L_g$ and $K_g$ containing $P_g$ and
$R_g$, respectively, having simultaneously a robust
heterodimensional cycle and a robust homoclinic tangency.
\end{coro}

 Theorem~\ref{t.main} has the following consequence
 for
$C^1$-generic diffeomorphisms of three dimensional manifolds that
slightly  generalizes the dichotomy
 ``domination versus infinitely many sources/sinks" in \cite{BDP:03}.

\begin{coro}\label{c.main2}
Let $M$ be a closed manifold of dimension three. There is a
residual subset $\cR$ of $\diff$ such that for every
diffeomorphism $f$ and every saddle $P$ of stable index $2$ of $f$
(at least) one of the following three possibilities holds:
\begin{itemize}
\item
$H(P,f)$ has a dominated splitting;
\item
$H(P,f)$ is the Hausdorff limit of  periodic sinks;
\item $f$ has a robust heterodimensional  cycle associated to $P$ and
$H(P,f)$ is the Hausdorff limit of  periodic sources.
\end{itemize}
\label{t.trichotomy}
\end{coro}

\subsection{Non-domination far from heterodimensional cycles implies sectional dissipativity.}\label{ss.farfromheterodimnsional}
One approach for settling Palis conjecture is to study dynamics
{\emph{far from homoclinic tangencies.}} In this case the
diffeomorphisms necessarily have nice dominated splittings that
are adapted to their index structure, see for instance
\cite{W:04}. In contrast, dynamics {\emph{far from
heterodimensional cycles}} is yet little understood. To address
this point we will make the following ``local version" of
Conjecture~\ref{c.bonatti} where a given homoclinic class is
specified.
\smallskip

\noindent {\bf Conjecture \ref{c.bonatti}'.}
 \emph{Let  $P$ be a hyperbolic saddle of a
diffeomorphism $f$ such that for  every diffeomorphism $g$ that is
$C^1$-close to $f$ there is no heterodimensional cycle associated
to the continuation $P_g$ of $P$. Then  there exists a
diffeomorphism $g$ arbitrarily $C^1$-close to $f$ such that the
homoclinic class $H(P_g,g)$ is hyperbolic.}

\smallskip

To discuss Conjecture~\ref{c.bonatti}' let us first consider a
simple illustrating case involving the notion of {\emph{sectional
dissipativity}}. Let $P$ be a hyperbolic saddle of a
diffeomorphism $f$ of stable index $1$ whose homoclinic class
$H(P,f)$ satisfies the following two properties:
\begin{itemize}
\item
 $H(P,f)$ has no dominated splitting of index $1$ and
\item $H(P,f)$ is
uniformly sectionally dissipative for $f^{-1}$, that is, there is
$n>0$ such that the Jacobian of $f$ in restriction to any
$2$-plane is strictly larger than $1$.
\end{itemize}
Under these hypotheses, the lack of domination of $H(P,f)$
corresponding to the index of $P$ enables a homoclinic tangency
associated to $P$ after a perturbation. However, the sectional
dissipativity prevents the existence of saddle points of stable
index larger than $1$ in a small neighborhood of the homoclinic
class of $P$. Thus any diffeomorphism $g$ that is $C^1$-close to
$f$ cannot have a heterodimensional cycle associated to $P_g$.

We wonder if the case above is the only possible setting where
homoclinic tangencies far from heterodimensional cycles can occur.
We provide a partial result to this question by considering
homoclinic classes without any dominated splitting and a weaker
notion of sectional dissipativity.

Consider a set of periodic points $\cP$ of a diffeomorphism $f$
and a $Df$-invariant subbundle $E$ defined over the set $\cP$. The
bundle $E$ is said to be {\emph{sectionally dissipative at the
period}} if for any point $R\in \cP$ there is a constant $0<
\alpha_R<1$ such that $|\la_k\,\la_{k+1}|<\alpha_R^{\pi(R)}$ for
every pair of multipliers $\la_k$ and $\la_{k+1}$ of $R$ whose
eigendirections are contained in $E$. When $E=T_\cP M$  then we
call the set of periodic points $\cP$ {\emph{sectionally
dissipative at the period.}} In the case that the constant
$\alpha_R$ can be chosen independently of $R$ we call the bundle
$E$ (or the set $\cP$) {\emph{uniformly sectionally dissipative at
the period.}}

\begin{coro} \label{c.excnodominated} Let $M$ be a closed manifold $M$ with $dim(M)\geq 3$
and $f\colon M\to M$ a diffeomorphism. Consider a homoclinic class
$H(P,f)$  without any dominated splitting that is far from
heterodimensional cycles. Then the set of periodic points of $f$
homoclinically related to $P$ is uniformly sectionally dissipative
at the period either for $f$ or for $f^{-1}$.
\end{coro}


\subsection{Robust non-domination}\label{ss.robust}
We first recall that the existence of a dominated splitting is (in
some sense) an open property. More precisely, if $\La$ is an
$f$-invariant compact set with a dominated splitting $T_\La M
=E\oplus F$, then there are neighborhoods $U$ of $\La$ in $M$ and
$\cU$ of $f$ in $\diff$ such that for every $g\in \cU$ and every
$g$-invariant set $\Sigma$ contained in $U$ there is a dominated
splitting for $\Sigma$ of the same index as $E\oplus F$, see for
instance \cite[Chapter B.1]{BDV:04}. Observe that the next theorem
implies that, in some cases, the absence of domination of a
homoclinic class can, after a perturbation, be turned into a
robust property.

\begin{theo}
\label{t.complex} Let $H(P,f)$ be a non trivial homoclinic class
of a periodic point $P$ of stable index $i$. Assume that for some
$j\ne i$ there is no dominated splitting of index $j$. Then there
exists a diffeomorphism $g$ arbitrarily $C^1$-close to $f$ having
a periodic point $Q$ that is homoclinically related to $P_g$ and such that
$\la_j(Q)$ and $\la_{j+1} (Q)$ are non-real, have the same
modulus, and any $k$-th multiplier of $Q$ has modulus different
from $|\la_j(Q)|$, ($k\ne j,j+1$).
\end{theo}

An immediate consequence of this theorem is that for every
diffeomorphism $h$ close to $g$ the homoclinic class $H(P_h,h)$
does not have a dominated splitting of index $j$.

A more detailed version of this theorem is given in
Proposition~\ref{p.complex}. Unfortunately, it still remains to
settle the hardest case in which the lack of domination of the
class $H(P,f)$ corresponds to the stable index of $P$.

Observe that, under the hypotheses of Theorem~\ref{t.complex}, the
constructions in \cite{BDP:03} imply that there are points $Q$
homoclinically related to $P$ whose multipliers $\la_j(Q)$ and
$\la_{j+1}(Q)$ can be made non-real by small perturbations. The
difficulty in the theorem is to preserve the homoclinic relation
between $P$ and $Q$ throughout the perturbation process.

The following result is a consequence of Theorem~\ref{t.complex}
and the fact that for $C^1$-generic diffeomorphisms two saddles in
the same chain recurrence class robustly belong to the same chain
recurrence class
 (see Section~\ref{s.robustizing} for the proof).

\begin{coro} \label{c.complexb}
There is a residual set $\cG$ of $\diff$ such that for every $f\in
\cG$ and every homoclinic class $H(P,f)$ of $f$ having periodic
points of different stable indices the following holds:

\noindent if the class $H(P,f)$ has no dominated splitting of
index $j$ then for any diffeomorphism $g$ in a neighborhood of $f$
the chain recurrence class of $P_g$ has no dominated splitting of
index $j$.
\end{coro}

\subsection{Robust non-domination and self-replication}\label{ss.viral}
In \cite[Definition 1.1]{BD:02}, for diffeomorphisms defined on
three-dimensional manifolds, we consider the following open
property for chain recurrence classes that we call \emph{Property
$\fC$}.

\begin{itemize}
\item[(i)] The class contains two transitive hyperbolic
sets $L$ and $K$ of different stable indices related by a robust
heterodimensional cycle.
\item[(ii)] Each of these sets $K$, $L$
contains a saddle with non-real multipliers.
\item[(iii)] Each of these sets contains a saddle whose
Jacobian is greater than one and a saddle whose Jacobian is less
than one.
\end{itemize}

A key ingredient in \cite{BD:02} is the notion of {\emph{universal
dynamics:}} Given a diffeomorphism $f$ with Property $\fC$ by
perturbation we can produce ``any type" of dynamics in a ball
isotopic to the identity (for large iterations of the
diffeomorphisms). In particular, after perturbations one can
re-obtain properties of any orientation preserving diffeomorphism
of a closed ball, see \cite[Definition 1.3]{BD:02}. As a
consequence, chain recurrence classes satisfying Property $\fC$
generate new different classes satisfying also this property. Thus
Property $\fC$ is a ``self-replicant" or ``viral" property. This
is the main motivation behind the definition of a viral property
in \cite[Sections 7.3-7.5]{B:bible}.

\begin{defi}[Viral property]
A property $\fQ$ of chain recurrence classes of saddles is said to
be {\emph{$C^k$-viral}} if for every diffeomorphism $f$ and every
saddle $P$ of $f$ whose chain recurrence class $C(P,f)$ satisfies
$\fQ$ the following conditions hold:

\smallskip

\noindent{\em{Robustness.}}
 There is a $C^k$-neighborhood $\cU$ of $f$ such that $C(P_g,g)$ also satisfies $\fQ$ for all $g\in \cU$.

\smallskip

\noindent{\em{Self-replication.}} For every $C^k$-neighborhood
$\cV$ of $f$ and for every neighborhood $V$ of $C(P,f)$ there are
a diffeomorphism $g\in \cV$ and a hyperbolic periodic point
$Q_g\in V$ of $g$ such that $C(Q_g,g)$ is different from
$C(P_g,g)$ and satisfies property $\fQ$. \label{d.bviral}
\end{defi}

As observed above, the existence of a robust homoclinic tangency
(associated to a transitive hyperbolic set in the class) is an
example of a $C^2$-viral property for chain recurrence classes in
dimension two.

As a consequence of the above results we now confirm
\cite[Conjecture 14]{B:bible} claiming that the property of robust
non-existence of any dominated splitting over a chain recurrence
class of a saddle is viral in the case that the class contains
saddles whose stable indices are different from $1$ and
$\dim(M)-1$. We formulate the following generalization of
Property~$\fC$.

\begin{defi}[Property $\fV$] \label{d.propertyS}
Given a saddle $P$ of a diffeomorphism $f$, the chain recurrence
class
 $C(P,f)$ of $P$ satisfies
\emph{Property $\fV$} if there is a $C^1$-neighborhood $\cU$ of
$f$ such that for all $g\in \cU$ the chain recurrence class
$C(P_g,g)$ of $P_g$ satisfies the following two conditions:
\begin{itemize}
\item {\rm (non-domination)}
$C(P_g,g)$
 does not have any dominated splitting,
\item {\rm (index variability)}
$C(P_g,g)$ contains a saddle $Q_g$ whose stable index is different
from the one of $P_g$.
\end{itemize}
\end{defi}

Observe  that the set of $C^1$-diffeomorphisms satisfying
Property~$\fV$ is indeed non-empty, see Section~\ref{ss.examples}.

\begin{theo}\label{t.bviral}
Property $\fV$ is $C^1$-viral  for chain recurrence classes.
\end{theo}

The following result is a consequence of Theorem~\ref{t.bviral}
and the properties of $C^1$-generic diffeomorphisms extending
\cite{BD:02}. In fact, the corollary holds for any viral property
of a chain recurrence class containing a saddle.

\begin{coro}
\label{c.bviral} Let $C(P,f)$ be a chain recurrence class
satisfying Property~$\fV$. Then there are a neighborhood $\cU$ of
$f$ and a residual subset $\cG_\cU$ of $\cU$ such that for every
$g\in \cG_{\cU}$
\begin{itemize}
\item
there are infinitely (countably) many pairwise disjoint homoclinic
classes, and
\item
there are uncountably many aperiodic chain recurrence classes.
\end{itemize}
\end{coro}

Indeed the homoclinic classes obtained in the corollary can be
chosen to also satisfy Property~$\fV$.
 The
proofs of Theorem~\ref{t.bviral} and Corollary~\ref{c.bviral} are
in Section~\ref{s.viral}.

Let us observe that nature of the proof of Theorem~\ref{t.bviral}
is quite different from the approach in \cite{BD:02}, where
universal dynamics is the key ingredient. In \cite{BD:02} this
universal dynamics is obtained by considering saddles in the chain
recurrence class whose Jacobians are larger and smaller than one,
respectively. A restriction of this construction is that all
Lyapunov exponents of the aperiodic classes obtained in
\cite{BD:02} are zero. This follows from the fact that one
considers maps whose ``returns" are close to the identity. Here we
use directly the self-replication property.  This allows us to
obtain aperiodic classes with regular points having Lyapunov
exponents uniformly bounded away from zero. See \cite[Section
7.4]{B:bible}, specially Problem 6, for further discussion.

 Finally,
bearing in mind the results in \cite{S:pre} and
Corollary~\ref{c.main}, we introduce the following variation of
Property~$\fV$ for
 diffeomorphisms defined on manifolds of dimension $d\ge 4$.

\begin{defi}[Property $\fV'$] \label{d.propertySprime}
Given a saddle $P$ of a diffeomorphism $f$  the chain recurrence
class
 $C(P,f)$ of $P$ has
\emph{Property $\fV'$} if there is a $C^1$-neighborhood $\cU$ of
$f$ such that for all $g\in \cU$ the chain recurrence class
$C(P_g,g)$ of $P_g$ satisfies the following two conditions:
\begin{itemize}
\item
$C(P_g,g)$
 does not have any dominated splitting and
\item
$C(P_g,g)$ contains a saddle with  stable index $i\not\in\{1,
\dim (M)-1\}$.
\end{itemize}
\end{defi}

Corollary~\ref{c.main} implies that in this case, after a
perturbation, the chain recurrence class $C(P_g,g)$ robustly
satisfies the index variability condition. Thus, after a
perturbation, Property~$\fV'$ implies Property~$\fV$. In fact, we
will see that these two properties are ``essentially equivalent",
see Lemmas~\ref{l.VimpliesS} and~\ref{l.SimpliesV}. Finally, we
have the following:

\begin{rema}\label{r.ss'}
Theorem~\ref{t.bviral} and Corollary~\ref{c.bviral} hold for
Property~$\fV'$.
\end{rema}

\subsection*{Organization of the paper}\label{ss.organization}
We first observe that we will use systematically several
$C^1$-perturbation results imported from \cite{G:10,G:pre,BB:pre}.
These results allow us to realize dynamically perturbations of
cocycles associated to the derivatives of diffeomorphisms along
periodic orbits (see specially Section~\ref{ss.gfranks} and
\ref{ss.dominatedperiodic}).

\smallskip

\noindent $\bullet$
In Section~\ref{s.homohetero}  we recall results about the
generation of homoclinic tangencies and heterodimensional cycles
associated to homoclinic classes.

\smallskip

\noindent $\bullet$ An ingredient of our paper is the notion of an
{\emph{adapted perturbation}} of a diffeomorphism, that is, a
small perturbation of a
 diffeomorphism throughout the orbit of  a periodic point
that preserves some  homoclinic relations and some prescribed
dynamical properties of a given homoclinic class, (see
Definition~\ref{d.adapted}). An essential feature of adapted
perturbations is that one can perform simultaneously finitely many
of them preserving some prescribed properties of the homoclinic
class. These perturbations are introduced in
Section~\ref{s.adaptedfranks}.

\smallskip

\noindent $\bullet$
Using adapted perturbations we prove in
Section~\ref{s.lyapunovperiodic} two important technical results
(Propositions~\ref{p.gbobo} and \ref{p.weak}) claiming that the
 lack of domination of a homoclinic class yields periodic orbits having
multiple  Lyapunov exponents and weak hyperbolicity.

\smallskip

\noindent $\bullet$ In Sections~\ref{s.robustizing} and
\ref{s.proofof}, in the non-dominated setting we get periodic
orbits inside a homoclinic class having non-real multipliers and
prove Theorem~\ref{t.complex}. This proof is based on
Proposition~\ref{p.complex} whose proof is the most difficult step
of the paper.

\smallskip

\noindent $\bullet$ In Section~\ref{s.formation} we obtain
homoclinic intersections associated to strong invariant manifolds
of periodic points that will allow us to get heterodimensional
cycles  and finally prove Theorem~\ref{t.main} in
Section~\ref{s.proofofmain}.

\smallskip

\noindent $\bullet$
Finally, we study viral properties of chain recurrence classes and
prove Theorem~\ref{t.bviral} and Corollary~\ref{c.bviral} in
Section~\ref{s.viral} .


\section{Homoclinic tangencies and heterodimensional cycles}
\label{s.homohetero}

In this section we recall some results about generation of
homoclinic tangencies and robust heterodimensional cycles
associated to homoclinic classes.

\subsection{Homoclinic tangencies}\label{ss.homoclinictangencies}
Next lemma states the relation between the lack of domination over
a periodic orbit and the generation of homoclinic tangencies.

\begin{lemm}[{\cite[Theorem 3.1]{G:10}}]\label{l.gdcds}
For any $K>1$, $\pes>0$, and  $d\in \NN$, there are
constants $k_0$ and $\ell_0$ with the following property.
\begin{itemize}
\item
For every $f\in \diff$ with $\dim (M)=d$ such that the norms of $Df$ and $Df^{-1}$
are both bounded by $K$, and
\item
for every  periodic point $P$ of $f$ of saddle-type such that
\begin{itemize}
\item the period of $P$ is larger than $\ell_{0}$ and
\item
the stable/unstable splitting $E^s(f^i(P))\oplus E^u(f^i(P))$ over
the orbit of $P$ is not $k_{0}$-do\-mi\-na\-ted,
\end{itemize}
\end{itemize}
there is an $\pes$-perturbation $g$ of $f$ whose support is
contained in an arbitrarily small neighborhood of the orbit of $P$
and such that the stable and unstable manifolds $W^s(P,g)$ and
$W^u(P,g)$ of $P$ have a homoclinic tangency.

Moreover, if $Q$ is homoclinically related to $P$ for $f$ then the
perturbation $g$ can be chosen such that $Q_g$ and $P$ are
homoclinically related (for $g$).
\end{lemm}

\begin{rema}
\label{r.gdcds} Lemma~\ref{l.gdcds} implies that the perturbation
$g$ of $f$ can be chosen such that the saddle $P$ has a homoclinic
tangency and its homoclinic class $H(P,g)$ is non-trivial.
Moreover, the orbit of  tangency can be chosen inside the
homoclinic class $H(P,g)$.
\end{rema}

\subsection{Robust heterodimensional cycles}\label{ss.robustcycles}

Let us observe that a homoclinic class $H(P,f)$ may contain
saddles of different indices. But, in principle, it is not
guaranteed that such a  property still holds for perturbations
of $f$. We next collect some results from \cite{BDK:pre} that will
allow us to get such a  property in a robust way.

We say that a heterodimensional cycle associated to a pair of
transitive hyperbolic sets has {\emph{coindex one}} if the
$s$-indices of theses sets differ by one.

\begin{lemm}[\cite{BDK:pre}]
\label{l.bodiki} Let $f\in \diff$ be a diffeomorphism having
a coindex one
heterodimensional cycle
associated to a
pair of hyperbolic periodic points $P$ and $Q$ such that the homoclinic class
$H(P,f)$ is non trivial. Then there is a diffeomorphism $g$ arbitrarily $C^1$-close
to $f$ with a pair of hyperbolic transitive sets $L_g$ and $K_g$
having a robust heterodimensional cycle and containing the
continuations $P_g$ and $Q_g$ of $P$ and $Q$, respectively.
\end{lemm}

There is the following
 consequence of this lemma for $C^1$-generic systems:

\begin{corol}[\cite{BDK:pre}]
\label{c.bdk} There is a residual subset $\cG$ of $\diff$ such
that for every  diffeomorphism $f\in \cG$ and every pair of
periodic points $P$ and $Q$ of stable indices $i<j$ in the same
homoclinic class there is a (finite) sequence of transitive
hyperbolic sets $K_i,K_{i+1}\dots, K_j$ 
such that
\begin{itemize}
\item
$P\in K_i$, $Q\in K_j$, 
 \item 
the  stable index of  $K_n$ is  $n$, $n=i,i+1,\dots, j$,
and
\item
the sets
$K_k$ and $K_{k+1}$ have a robust heterodimensional cycle
for all $k=i,\dots, j-1$.
\end{itemize}
\end{corol}

\section{Adapted perturbations and generalized Franks' lemma} \label{s.adaptedfranks}

In this section, we collect some results about $C^1$-perturbations
of diffeomorphisms. Observe that if $g_1,\dots,g_n$ are
$\pes$-perturbations of $f$ with disjoint supports $V_1,\dots,V_n$
then there is an $\eps$-perturbation $g$ of $f$ supported in the
union of the sets $V_i$ such that $g$ coincides with $g_i$
over the set $V_i$.

\subsection{Adapted perturbations}\label{ss.adapted}

We next  introduce a kind of perturbation of 
diffeomorphisms along periodic orbits that preserves 
homoclinic relations. Moreover, these perturbations can be
performed simultaneously and independently along different
periodic orbits.

In what follows, given $\vro>0$, we denote by $W^{s,u}_\vro(P,f)$
the stable/unstable manifolds of size $\vro$ of the orbit of $P$.

\begin{defi}[Adapted perturbations]
Consider a property $\fP$ about periodic points. Given $f\in
\diff$, a pair of hyperbolic periodic points $P$ and $Q$ of $f$
that are homoclinically related,
and a neighborhood $\cU\subset \diff$ of $f$ 
 we say that there is a
\emph{perturbation of $f$ in $\cU$ along the orbit of $Q$ that is
adapted to $H(P,f)$ and property $\fP$} if
\begin{itemize}
\item
for  every neighborhood $V$ of the orbit of $Q$ and
\item
for every $\vro>0$ and every pair of compact sets $K^s\subset W^s_\vro(Q,f)$ and
$K^u\subset W^u_\vro(Q,f)$ disjoint from $V$
\end{itemize}
 there is a diffeomorphism $g\in \cU$ such that:
\begin{itemize}
\item $g$ coincides with $f$ outside $V$ and along the $f$-orbit of
$Q$,
\item the points $P_g$ and $Q_g$ are homoclinically related for $g$,
\item the sets $K^s,K^u$ are contained in $W^s_\vro(Q,g)$ and
$W^u_\vro(Q,g)$, respectively, and
\item the saddle $Q$ satisfies property $\fP$.
\end{itemize}

When the neighborhood $\cU$ of $f$ is the set of diffeomorphisms
that are $\pes$-$C^1$-close to $f$ we say that $g$ is an
\emph{$\pes$-perturbation of $f$ along the orbit of $Q$ that is
adapted to $H(P,f)$ and property $\fP$}.
 \label{d.adapted}
\end{defi}

Examples of property $\fP$ for  periodic points are the existence
of non-real multipliers and negative Lyapunov exponents.

\subsection{Generalized Franks' lemma}
\label{ss.gfranks}

We need the following extension of the so-called Franks Lemma
\cite{F:71} about dynamical realizations of
perturbations of cocycles along periodic orbits. The  novelty of
this extension is that besides the dynamical realization of the cocycle throughout a periodic orbit
the perturbations also preserve some homoclinic/heteroclinic intersections. Next lemma is a particular case of \cite[Theorem~1]{G:pre} and
is a key tool for constructing adapted perturbations. Recall that
a linear map $B\in GL(d,\RR)$ is {\emph{hyperbolic}} if every  eigenvalue $\la$ of $B$ satisfies $|\la|\ne 1$.

\begin{lemm}[Generalized Franks' Lemma, \cite{G:pre}]\label{l.gourmelon}
Consider $\pes>0$, a diffeomorphism $f\in \diff$ and a hyperbolic periodic point $Q$ of period $\ell=\pi(Q)$
of $f$. Then
\begin{itemize}
\item for any
one-parameter  family of linear maps
$(A_{n,t})_{n=0,\dots,\ell-1, \, t\in [0,1],}$, $A_{n,t}\in
\mathrm{GL}(d, \RR)$, $d=\dim (M)$, such that
\begin{enumerate}
\item \label{ifg1}
$A_{n,0}=Df (f^n(Q))$,
\item \label{ifg2}
for all $n=0,\dots,\ell-1$ and all $t\in [0,1]$ it holds
$$
\max \, \left\{ \| Df(f^n(Q)) -A_{n,t} \|, \, \| Df^{-1}(f^n(Q))
-A^{-1}_{n,t} \| \right\}<\varepsilon,
$$
\item \label{ifg3}
$B_t=A_{\ell-1,t}\circ\cdots\circ A_{0,t}$ is hyperbolic for all
$t\in [0,1]$, 
\end{enumerate}
\item
for  every neighborhood $V$ of the orbit of $Q$, every $\vro>0$,
and every pair of compact sets $K^s\subset W^s_\vro(Q,f)$ and $K^u
\subset W^u_\vro (Q,f)$ disjoint from $V$,
\end{itemize}
there is an $\pes$-perturbation $g$ of $f$ such that
\begin{enumerate}
\item[(a)] $g$ and $f$ coincide throughout the orbit of $Q$ and outside
$V$,
\item[(b)]
$K^s\subset W^s_\vro (Q,g)$ and $K^u \subset W^u_\vro(Q,g)$, and
\item[(c)]
$Dg(g^n(Q))=Dg(f^n(Q))=A_{n,1}$ for all $n=0,\dots,\ell-1$.
\end{enumerate}
\end{lemm}


\section{Lyapunov exponents of periodic orbits}
\label{s.lyapunovperiodic}

In this section we see that the lack of domination of a homoclinic
class yields perturbations such that there are periodic points of
the class whose Lyapunov exponents are multiple or close to zero,
see Propositions~\ref{p.gbobo} and \ref{p.weak}. We first state
some preparatory results and prove these propositions in
Section~\ref{ss.multiple}.

\subsection{Lyapunov exponents and homoclinic relations} \label{ss.homoclinic}
We will use repeatedly throughout  the paper the following result.

\begin{lemm}
\label{l.homocliniclyapunov} There is a residual subset $\cG$ of
$\diff$ such that for every $f\in \cG$, every saddle $P$ of $f$,
every non-trivial and locally maximal transitive hyperbolic set
$\La$ of $f$ containing $P$, and  every $\varepsilon>0$ there is a
saddle $Q\in \La$ such that
\begin{itemize}
\item
$|\chi_j(Q)-\chi_j(P)|<\varepsilon$ for all $j\in \{1,\dots,d\}$, and
\item
the orbit of $Q$ is $\varepsilon$-dense in $\La$. \end{itemize}
 In
particular, the saddle $Q$ can be chosen with arbitrarily large
period.
\end{lemm}

This results follows from the arguments in the proofs of
\cite[Corollary 2]{ABCDW:07} and \cite[Theorem 3.10]{ABC:} using
standard constructions that allow us to distribute these orbits
throughout the ``whole" transitive hyperbolic set while keeping
the control of the exponents.

\subsection{Dominated splittings and cocycles over periodic
orbits}\label{ss.dominatedperiodic}

We next  study the lack of domination of homoclinic classes. For
that we consider periodic orbits (of large period) in the class
and their associated cocycles. Next result is a
standard fact about dominated splittings (see for instance
\cite[Appendix B]{BDV:04}).

\begin{lemm}[Extension of a dominated splitting to a closure]
\label{l.dominatedclosure} Consider an $f$-invariant set $\La$
having a $k$-dominated splitting of index $i$. Then the
closure of $\La$ also has a $k$-dominated splitting of index $i$
 that coincides with the one over $\La$.
\end{lemm}

As in the case of periodic points of diffeomorphisms, given a
family of linear maps $A_1,\dots ,A_\ell\in GL(d,\RR)$ we consider the
product $B=A_\ell\circ\cdots \circ A_1$ and the eigenvalues
$\la_1(B),\dots,\la_d(B)$ of $B$ ordered in increasing modulus and
counted with multiplicity. We define the {\emph{$i$-th Lyapunov
exponent}} of $B$ by
$$
\chi_i(B)= \frac 1 \ell \log |\la_i(B)|.
$$

The family of linear maps above is \emph{bounded by $K$} if
$\|A_n\|$ and $\|A_n^{-1}\|$ are both less than or equal to $K$ for all 
$n=1, \dots, \ell$.

Note that Definition~\ref{d.dominated} of a dominated splitting
over an invariant set of a diffeomorphism can be restated for
sequences of linear maps.

\medskip

Next lemma  relates the lack of domination of a
cocycle and the generation of sinks or sources.

\begin{lemm}[{\cite[Corollary 2.19 and Remark 2.20]{BGV:06}}]
\label{l.bgv} For every $K>1$, $\pes>0$, and  $d\in \NN$,
there are constants $k_0$ and $\ell_0$ with the following property.
\begin{itemize}
\item
For every $f\in \diff$ with $\dim (M)=d$ such that the norms of $Df$ and $Df^{-1}$
are both bounded by $K$, and
\item
for every  periodic point $P$ of $f$ of period larger than $\ell_0$
such that there is no any $k_0$-dominated splitting over the orbit
of $P$,
\end{itemize}
 there is an $\pes$-perturbation $g$ of $f$ whose support  is contained in an arbitrarily small
neighborhood of the orbit of $P$ and  such that $P$ is either a
sink or a source of $g$.
\end{lemm}

Next result is a finer version of the previous lemma that allows us
to modify only two consecutive Lyapunov exponents of a cocycle.

\begin{lemm}[{\cite[Theorem~4.1 and Proposition 3.1]{BB:pre}}]\label{l.bobo}
For every $K>1$, $\varepsilon>0$ , and $d\geq 2$,
there are constants $k_0$ and $\ell_0$ with the following property.

Consider  $\ell\ge \ell_0$ and linear maps $A_1,\dots, A_{\ell}$ in $GL(d,\RR)$, such that:
\begin{itemize}
\item every $A_n$ is bounded by $K$,
\item for any $i\in \{1,\dots, d-1\}$,
the linear map $B=A_\ell\circ \cdots \circ A_1$ has no any $k_0$-dominated splitting of index $i$.
\end{itemize}
Then for every $j\in \{1,\dots,d-1\}$,
there exist one parameter families of linear maps $
(A_{n,t})_{t\in [0,1]}$ in $GL(d,\RR)$, $n=1,\dots,\ell$, such that
\begin{enumerate}
\item
$A_{n,0}=A_n$ for all $n=1,\dots,\ell$, and
\item
$A_{n,t}-A_n$ and $A_{n,t}^{-1}-A_n^{-1}$ are bounded by $\varepsilon$ for all $t\in [0,1]$ and all $n=1,\dots,\ell$.
\end{enumerate}
Consider the linear map
$$
B_t=A_{\ell,t}\circ \cdots \circ A_{1,t}.
$$
Then, for any $t\in [0,1]$, the Lyapunov exponents of the map $B_t$ satisfies
\begin{enumerate}
\item[(3)]
 $\chi_m(B_t)=\chi_m(B)$ if $m\ne j,j+1$,
\item[(4)]
$\chi_{j}(B_t)+\chi_{j+1}(B_t)= \chi_{j}(B)+\chi_{j+1}(B)$,
\item[(5)]
$\chi_{j}(B_{t^\prime})$ is non-decreasing and
$\chi_{j+1}(B_{t^\prime})$ is non-increasing, that is
$$
\chi_{j}(B_{t^\prime}) \le \chi_{j}(B_t) \le \chi_{j+1}(B_t) \le
\chi_{j+1}(B_{t^\prime}), \quad \mbox{for all $t'<t$},
$$
\item[(6)] 
$\chi_{j+1} (B_1)=\chi_j(B_1)$, and
\item[(7)] the eigenvalues of $B_1$ are all real.
\end{enumerate}
\end{lemm}

\begin{rema} \label{r.bobo}
Note that if $A\in GL(d,\RR)$ has real eigenvalues and if its
Lyapunov exponents $\chi_{j}(A)$ and $\chi_{j+1}(A)$ are equal
then there is $\bar A\in GL(d,\RR)$ arbitrarily close to $A$ whose
eigenvalues are real and whose Lyapunov exponents satisfy
$\chi_{m}(\bar A)\ne \chi_{j}(\bar A)=\chi_{j+1}(\bar A)$ for all
$m\ne j,j+1$. Moreover, there is a ``small path of cocycles"
joining $A$ and $\bar A$ that preserves the $j$ and $j+1$ Lyapunov
exponents.
Thus in the conclusions of Lemma~\ref{l.bobo} we can replace
item (3) by
\begin{enumerate}
\item[(3')]
$\chi_m(B_t)$ is close to $\chi_m(B)$ for all $m\ne j,j+1$ and all
$t\in [0,1]$ and $\chi_{m}(B_1)\ne \chi_{j}(B_1)=\chi_{j+1}(B_1)$.
\end{enumerate}
\end{rema}

In order to get cocycles with real eigenvalues we also use the following result
(see also previous results in \cite[Lemme 6.6]{BC:04}
and \cite[Lemma 3.8]{BGV:06}).

\begin{prop}[{\cite[Proposition 4.1]{BB:pre}}]\label{p.bobo}
For every $K>1$, $\varepsilon>0$ , and $d\geq 2$,
there is a constant $\ell_0$ with the following property.

Consider  $\ell\ge \ell_0$ and linear maps $A_1,\dots, A_{\ell}$ in $GL(d,\RR)$, such that:

For every family of linear maps $(A_n)_{n=1}^\ell$ in $GL(d,\RR)$ such that $\ell\ge
\ell_0$ and $A_n$ and $A_n^{-1}$ are bounded by $K$ for every $n$,
there are one parameter families of linear maps $(A_{n,t})_{n=1, t\in [0,1]}^\ell$,
in $GL(d,\RR)$, such that
\begin{itemize}
\item
 $A_{n,0}=A_n$,
\item
$A_{n,t}-A_n$
and $A_{n,t}^{-1}-A_n^{-1}$ are bounded by $\varepsilon$ for every $n$,
\item let $B_t=A_{\ell,t}\circ \cdots \circ A_{1,t}$, then for every $j\in
\{1,\dots,d\}$ the  Lyapunov exponent  $\chi_{j}(B_t)$ is constant for $t\in
[0,1]$, and
\item
all the multipliers of $B_1$ are  real.
\end{itemize}
\end{prop}

\subsection{Multiple Lyapunov exponents and weak hyperbolicity}
\label{ss.multiple}
In Pro\-po\-sitions~\ref{p.gbobo} and \ref{p.weak} 
we combine Lemmas~\ref{l.gourmelon} and \ref{l.bobo} to prove that the
lack of domination of a homoclinic class yields periodic orbits
whose Lyapunov exponents are multiple or close to zero.

\begin{prop} \label{p.gbobo}
For every $K>1$, $\pes>0$, and  $d\in \NN$,
there is a constant $k_0$ with the following property.

Consider a diffeomorphism $f\in \diff$, $\dim (M)=d$, such that the norms of
$Df$ and $Df^{-1}$ are bounded by $K$,
a hyperbolic periodic point $P$ of $s$-index $i$ 
whose homoclinic class $H(P,f)$ is non-trivial,
and an integer $j\in \{1,\dots,d\}$ with
$j\ne i$ such that the homoclinic class  $H(P,f)$ has
no any $k_0$-dominated splitting of index $j$.

Then there is a periodic point
$Q\in H(P,f)$ homoclinically related with $P$ and an $\varepsilon$-perturbation
$g$ of $f$ along the orbit of $Q$ that is adapted to
$H(P,f)$ and to the following property $\fP_{j,j+1}$:
$$
\fP_{j,j+1}\eqdef \begin{cases}
&\text{$\chi_j(Q_g)=\chi_{j+1}(Q_g)$,}\\
& \text{$\chi_m(Q_g)\neq \chi_j(Q_g)$ for all  $m\neq j,j+1$,}\\
& \text{$\la_m(Q_g)\in \RR$ for all $m$.}
\end{cases}
$$
\end{prop}
\begin{proof}
Consider 
the constants $d\in \NN$, $K>1$,  and $\varepsilon>0$. 
Applying Lemma~\ref{l.bobo} to these constants
we obtain the constants
$k_0$ and $\ell_0$.

Since the homoclinic class $H(P,f)$ is non-trivial, the set
$\Si_{\ell_0}$ of all saddles $Q$ of period larger than $\ell_0$
that are homoclinically related to $P$ is dense in $H(P,f)$.
Observe that there is a saddle $Q\in \Si_{\ell_0}$ such that there
is no $k_0$-dominated splitting of index $j$ over the orbit of
$Q$. Otherwise, by Lemma~\ref{l.dominatedclosure}, the closure of
the set $\Si_{\ell_0}$ (that is the whole class $H(P,f)$) would
have a $k_0$-dominated splitting of index $j$, which is a
contradiction.

Thus we can  apply Lemma~\ref{l.bobo} to the linear maps
$Df(Q),\dots, Df^{\ell-1}(Q)$, $\ell=\pi(Q)\ge \ell_0$, obtaining
one-parameter families of linear maps $(A_{i,t})_{t\in [0,1]}$,
$i=0,\dots,\ell-1$, satisfying the conclusions of
Lemma~\ref{l.bobo}.

We now fix a neighborhood $V$ of the orbit $Q$ and compact sets
$K^s \subset W^s(Q,f)$ and
$K^u\subset W^u(Q,f)$ disjoint from $V$
as in Definition~\ref{d.adapted}. Since  $Q$ is
homoclinically related to $P$ there are transverse intersections
$Y^s\in W^s(Q,f)\pitchfork W^u(P,f)$ and $Y^u\in
W^u(Q,f)\pitchfork W^s(P,f)$ and (small) compact disks
$\De^s\subset W^s(Q,f)$ and $\De^u\subset W^u(Q,f)$ (of the same
dimension as $W^s(Q,f)$ and $W^u(Q,f)$) containing  the points
$Y^s$ and $Y^u$. We consider the compact sets
 $$
 \tilde K^s=K^s\cup \De^s
\quad \mbox{and} \quad
 \tilde K^u=K^u\cup \De^u.
$$
We now apply Lemma~\ref{l.gourmelon} to $\varepsilon$,
$f$, the small path of
cocycles $(A_{n,t})$ above, and   the compact sets $\tilde
K^s$ and $\tilde K^u$ to get an $\varepsilon$-perturbation $g$ of
$f$ along the orbit of $Q$ adapted to $H(P,f)$ and Property $\fP_{j,j+1}$:
\begin{itemize}
\item
adapted to $H(P,f)$: By the choice of $\De^s$ and $\De^u$ the saddle $Q_g$ is
homoclinically related to $P_g$.
\item
adapted to
 Property $\fP_{j,j+1}$: By item (6) in
 Lemma~\ref{l.bobo} it holds
 $\chi_j(B_1)=\chi_{j+1}(B_1)$,
by Remark~\ref{r.bobo} we have $\chi_m(B_1)\ne \chi_j(B_i)$ if $m\ne
j,j+1$,
  and by item (7) all the eigenvalues of $B_1$ all are real.
\end{itemize}
  This concludes the proof of the
 proposition.
\end{proof}

\begin{prop}\label{p.weak}
For every $K>1$, $\pes>0$, and  $d\in \NN$,
there is a constant $k_0$ with the following property.

Consider $\delta>0$, a diffeomorphism $f\in \diff$, $\dim (M)=d$,  and a
hyperbolic periodic point $P$ of $f$ of $s$-index $i$ such that:
\begin{itemize}
\item
the norms of
$Df$ and $Df^{-1}$ are bounded by $K$,
\item
$\chi_i(P)+\chi_{i+1}(P)>-\delta$,
\item
the homoclinic class  $H(P,f)$ is non-trivial and has no $k_0$-dominated
splitting of index $i$.
\end{itemize}
Then there is a periodic point
$Q\in H(P,f)$ homoclinically related with $P$ and an $\varepsilon$-perturbation
$g$ of $f$ along the orbit of $Q$ that is adapted to
$H(P,f)$ and to the following property 
\begin{equation}
\label{e.pid}
 \fP_{i,\de}\eqdef
\text{The $i$-th Lyapunov exponent of $Q$ satisfies $\chi_i(Q)\in
(-\de,0)$}.
\end{equation}
\end{prop}

\begin{proof} The strategy of the proof is analogous to the one of
Proposition~\ref{p.gbobo}, so we will skip some repetitions.
As in the proof of Proposition~\ref{p.gbobo} we consider
 constants $k_0$ and $\ell_0$ associated to $K$, $\ve$, and $d$.

Since the homoclinic class $H(P,f)$ has no dominated splitting of
index $i$, there is a locally maximal transitive hyperbolic subset
$L$ of $H(P,f)$ containing $P$ and having no $k_0$-dominated
splitting of index $i$. We can also assume  that for every $f'$
close to $f$ the continuation $L_{f'}$ of $L$ has no such a
$k_0$-dominated splitting. 

We choose $f'$ in the residual subset
$\cG$ of $\diff$ in Lemma~\ref{l.homocliniclyapunov}. Then there is a
periodic point $Q_{f'}\in L_{f'}$ such that $\chi_i(Q_{f'})+
\chi_{i+1}(Q_{f'})>-\de$ and whose orbit has no $k_0$-dominated
splitting of index $i$. Otherwise, by Lemma~\ref{l.dominatedclosure},
the set $L_{f'}$ has a $k_0$-dominated splitting. 

Consider the point $Q=Q_f$. We take a first small path of
hyperbolic cocycles   $(\bar A_{n,t})_{t\in [0,1]}$,
$n=0,\dots,\ell-1$, $\ell=\pi(Q)$, over the orbit of $Q$ joining the
derivatives $Df$ and $Df'$. Note that, by definition, the cocycle $(\bar A_{n,1})$
does not have a $k_0$-dominated splitting and
the Lyapunov exponents of $\bar B_1= \bar A_{\ell-1,1}\circ \cdots
\circ \bar A_{0,1}$ satisfy $\chi_i(\bar B_1)+ \chi_{i+1}(\bar
B_1)>-\de$.


Observe that if $\chi_i(\bar B_1)>-\de$ we are done. Otherwise we
apply  Lemma~\ref{l.bobo} to the cocycle $\bar A_{n,1}$,
$n=0,\dots, \ell-1$, and $j=i$. This provides new  families of
linear maps $(\tilde A_{n,t})_{t\in [0,1]}$, $n=0,\dots,\ell-1$,
satisfying the conclusions of the lemma. Define the composition
$\tilde B_t$ as above. Let
$$
\tau \eqdef \chi_i(\tilde B_t)+\chi_{i+1}(\tilde B_t)>-\de.
$$
Note that by item (4) of Lemma~\ref{l.bobo} this number does not depend on $t$.

By item (6) in Lemma~\ref{l.bobo}, there is some $t_0$ such that
$$
\chi_i(\tilde B_{t_0})=\min \left(\frac{\tau -\de}2,\frac {-\delta} 2 \right).
$$ 
As the map
$\chi_i(\tilde B_{t})$ is non-decreasing (recall item (5) in
Lemma~\ref{l.bobo}) we have $\chi_i(\tilde B_{t})\le
\frac{-\de}2<0$ for all $t\in [0,t_0]$.
Also 
$$ \chi_{i+1}(\tilde B_t)\ge \tau - \min\left(\frac{\tau -\de}2,\frac {-\delta} 2\right)
\ge \frac{\tau+\delta} 2+\frac{\max(0,\tau)}{2}>0.$$
Therefore $(\tilde A_{n,t})_{n,t\in[0,t_0]}$ is a path of hyperbolic cocycles.

We next consider the concatenation of the paths of hyperbolic cocycles
$(\bar A_{n,t})_{t\in [0,1]}$ and $(\tilde A_{n,t})_{t\in
[0,t_0]}$.
The end of the proof is the same as the one of
Proposition~\ref{p.gbobo} and involves the definition of the sets
$\tilde K^s$ and $\tilde K^u$.
We apply Lemma~\ref{l.gourmelon} to get an $\varepsilon$-perturbation $g$ of $f$
along the orbit of $Q$ that is adapted to $H(P,f)$
 and to property $\fP_{i,\de}$, since by construction
$$ \chi_{i}(Q_g)=\chi_{i}(\tilde B_{t_0})=-\delta+
\min\left(\frac{\tau+\de}{2},\frac \delta 2\right)> -\de.$$
This ends the proof of the proposition.
\end{proof}
%
%


\section{``Robustizing" lack of domination} \label{s.robustizing}

In this section we analyze the existence of dominated splittings
for homoclinic classes. In some cases these splittings will have
several bundles.

\begin{defi}[Dominated splittings II]
\label{d.severalbundles} Let $\La$ be an invariant set of a
diffeomorphism $f$.
A $Df$-invariant splitting $E_1\oplus \cdots \oplus E_s$, $s\ge 2$, over the set
$\La$ is dominated if for all
$j\in \{1,\dots,s-1\}$ the splitting $E_1^j\oplus E_{j+1}^s$ is
dominated, where $E_1^j=E_1\oplus \cdots \oplus E_j$ and
$E_{j+1}^s=E_{j+1}\oplus \cdots \oplus E_s$.

As in the case of two bundles, the splitting is
\emph{$k$-dominated} if the splittings   $E_1^j\oplus E_{j+1}^k$
are $k$-dominated for all $j$.

There are analogous definitions for cocycles.
\end{defi}

Note that if there is a saddle $Q$ homoclinically related to $P$
such that $\chi_j(Q)=\chi_{j+1}(Q)$  then the class has no
dominated splitting of index $j$. Moreover, if
$$
\chi_{j-1}(Q)< \chi_j(Q)=\chi_{j+1}(Q)<\chi_{j+2}(Q) \quad
\mbox{and} \quad \la_j(Q), \la_{j+1}(Q)\in (\CC \setminus \RR)
$$
 then the lack of domination of the homoclinic class
 is $C^1$-robust.
In this section we study when the converse holds (up to
perturbations).

A  saddle $Q$ of a diffeomorphism $f$ satisfies
property $\fP_{j,j+1,\CC}$ if

\begin{equation} \label{e.complex}
\fP_{j,j+1,\CC}\eqdef
\begin{cases}
&\mbox{\textbf{(i)} $\chi_j(Q)=\chi_{j+1}(Q)$,}\\
&\mbox{\textbf{(ii)} $\chi_m(Q)\neq
\chi_j(Q)$ for all  $m\neq j,j+1$,}\\
&\mbox{\textbf{(iii)} $\la_j(Q)$ and $\la_{j+1}(Q)$ are non-real.}
\end{cases}
\end{equation}

The main technical step of our constructions is the next
proposition whose proof is postponed to the next section.
It immediately implies Theorem \ref{t.complex}.

\begin{prop}\label{p.complex}
For any $K>1$, $\pes>0$, and  $d\in \NN$, there is
a constant $k_0$ with the following property.

Consider a diffeomorphism $f\in \diff$, $\dim M=d$, such that the norms of
$Df$ and $Df^{-1}$ are bounded by $K$, a hyperbolic periodic
point $P$ of $s$-index $i$, and an integer $j\in \{1,\dots,d-1\}$, $j\ne i$.
Assume that the homoclinic class $H(P,f)$ is non trivial and has
no $k_{0}$-dominated splitting of index $j$.

Then there is a periodic point $Q$ that is homoclinically related with $P$ and an
$\varepsilon$-perturbation of $f$ along the orbit of $Q$ that is
adapted to $H(P,f)$ and property $\fP_{j,j+1,\CC}$.
\end{prop}

\begin{rema}\label{r.complex}
The proof of the proposition provides a point $Q$ with arbitrarily
large period. In particular, there exist infinitely many periodic
points $Q$ satisfying the conclusion of the proposition.
\end{rema}

We postpone the proof of this proposition to
Section~\ref{s.proofof}. We now deduce from it
Corollaries~\ref{c.1} and \ref{c.noname} below.

\begin{corol}\label{c.1}
For any $K>1$, $\pes>0$, and  $d\in \NN$, there is
a constant $k_0$ with the following property.

Consider a diffeomorphism $f\in \diff$, $\dim M=d$, such that the norms of
$Df$ and $Df^{-1}$ are bounded by $K$, a homoclinic class $H(P,f)$ of $f$,
and integers $0<j_1<\dots<j_\ell<d$  that are different from the
$s$-index of $P$ and such that
there is no $k_0$-dominated splitting of index $j_k$ over $H(P,f)$ for every
$k\in \{1,\dots,\ell\}$.

Then there exists an $\varepsilon$-perturbation $g$ of $f$ supported in a small
neighborhood of $H(P,f)$ such that for
each $k\in \{1,\dots,\ell\}$ there exists a periodic point
$Q_{k,g}$ of $g$ homoclinically related to $P_g$ satisfying property
$\fP_{j_k,j_k+1,\CC}$ in equation~\eqref{e.complex}.

In particular, for every diffeomorphism $\bar g$ close to $g$ and for
every $k\in \{1,\dots,\ell\}$ there is no dominated splitting of
index $j_k$ over $H(P_{\bar g},\bar g)$.
\end{corol}

\begin{proof} By
Proposition~\ref{p.complex}, for each index $j_k$ there is a
periodic point $Q_k$ homoclinically related to $P$ and $\pes$-perturbations of $f$ along the orbit of $Q_k$ that are
adapted to $H(P,f)$ and to property $\fP_{j_k,j_k+1,\CC}$. For
each saddle $Q_k$ consider a pair of transverse heteroclinic
points
$$
Y_k^s \in W^s(Q_k,f)\pitchfork W^u(P,f) \quad \mbox{and} \quad
Y_k^u \in W^u(Q_k,f)\pitchfork W^s(P,f).
$$
For each $k$ we also fix compact disks
$$
K_k^s \subset W^s(Q_k,f) \quad \mbox{and} \quad K_k^u \subset
W^u(Q_k,f)
$$
of the same dimensions as $W^s(Q_k,f)$ and $W^u(Q_k,f)$ containing
$Y_k^s$ and $Y_k^u$ in their interiors. By Remark~\ref{r.complex},
we can assume that the orbits of the saddles $Q_k$ are different.
Thus there are small neighborhoods $V_1,\dots, V_\ell$ of these orbits whose
closures are pairwise disjoint and  such that for each $k\ne k'$
the orbits of $Y_k^s$ and $Y_k^u$ do not intersect $V_{k'}$. Thus
taking the disks $K^s_k$ and $K_k^u$ small enough, we can assume
that this also holds for the forward orbit of $K_k^s$ and the
backward orbit of $K_k^u$.

For each $k$ we get an adapted $\pes$-perturbation $g_k$ supported
in $V_k$ (and associated to the compact sets $K^s_k$ and $K^u_k$).
Since the supports of these perturbations are disjoint, we can
perform all them simultaneously obtaining a diffeomorphism $g$
that is $\pes$-close to $f$ and has saddles $Q_{k,g}$ satisfying
$\fP_{j_k,j_k+1,\CC}$, $j=1,\dots,\ell$.

It remains to check that these saddles are homoclinically related
to $P_g$. Observe that for each $k$ the points $Y_k^s$ and $Y_k^u$
are transverse heteroclinic points (associated to $Q_k$ and $P$)
for $g_k$.  The choices of the orbits of these heteroclinic points
and of the sets $V_j$ imply that $Y_k^s$ and $Y_k^u$ are also
transverse heteroclinic points (associated to $Q_k$ and $P$) for
$g$ (in fact, the orbits of the points $Y_k^s$ and $Y_k^u$ are the
same for $g_k$ and $g$). This completes the proof of the
corollary.
\end{proof}

We also get the following genericity result.

\begin{corol} \label{c.noname}
There exists a residual subset  $\cG$ of  $\diff$ such that every
diffeomorphism $f\in \cG$ satisfies the following property:

For every  $i,j\in\{1,\dots,d-1\}$, $i\ne j$, and for every
periodic point $P$ of $s$-index $i$ of $f$ such that
 there is no
dominated splitting of index $j$ over $H(P,f)$ there exists a
periodic point $Q$ homoclinically related to $P$ satisfying
property $\fP_{j,j+1,\CC}$.
\end{corol}

The corollary follows from standard genericity arguments after
noting that for a homoclinic class $H(P,f)$ to have a saddle $Q$
homoclinically related to $P$ satisfying property
$\fP_{j,j+1,\CC}$ is an open condition.

\medskip

We are now ready to prove Corollary~\ref{c.complexb}.

\begin{proof}[Proof of Corollary~\ref{c.complexb}]
The residual subset $\cG$ in Corollary~\ref{c.noname} can be
chosen with the following additional property, see \cite{BC:04}.
For every $f\in\cG$ and for every pair of hyperbolic periodic points
$P$ and $Q$ of $f\in \cG$ that are in the same chain recurrent class the following holds
\begin{itemize}
\item
the homoclinic classes of $P$ and $Q$ are equal and
\item
there is a neighborhood $\cU$ of $f$ such that for all $g\in \cU$
the chain recurrence classes of $P_g$ and $Q_g$ are equal.
\end{itemize}
Now it is enough to consider a point $Q\in H(P,f)$ of $s$-index
different from the one of $P$ and to apply
Corollary~\ref{c.noname} to $P$ (if $j$ is different to the index
of $P$) or to $Q$ (otherwise).
\end{proof}

\medskip

\noindent{\emph{Comment.}} We wonder if in the conclusion
of Corollary~\ref{c.complexb} it is possible to consider
homoclinic classes instead of chain recurrence classes. One
difficulty is that in general one may have two hyperbolic periodic
points with different stable index  that are robustly in the same
chain recurrence class but whose homoclinic classes do not
coincide robustly. More precisely:

\begin{ques}
\label{q.=} Consider an open set $\,\cU$ of $\diff$ and two
hyperbolic saddles $P_f$ and $Q_f$ whose continuations are defined for all $f\in
\cU$, have different stable indices, and whose chain recurrence
classes coincide for all $f\in \cU$.

Does there exist an open and dense subset $\cV$ of $\cU$ such that
for any $f\in \cV$ one has $Q_f\in H(P_f,f)$? Or even more, $H(P_f,f)=H(Q_f,f)$?
\end{ques}

By \cite{BC:04} the answer to this question is affirmative when
the saddles have the same index. It is also true when the chain
recurrence class is partially hyperbolic with a  central direction
that splits into one-dimensional central directions. This
follows using quite standard arguments and we will provide the details
of this construction in a forthcoming note.

\section{Obtaining non-real multipliers: Proof of Proposition~\ref{p.complex}} \label{s.proofof}

In this section we prove Proposition~\ref{p.complex}. This
proposition follows from the next lemma:

\begin{lemm}\label{l.pcomplex}
Consider a homoclinic class $H(P,f)$ and $j\in \NN$ satisfying the
hypothesis of Proposition~\ref{p.complex}. Then there are a hyperbolic
periodic point $Q$  homoclinically related to $P$ and path
of cocycles $(A_{i,t})_{t\in [0,1]}$, $0\leq i<\ell$ and $\ell=\pi(Q)$,
over the orbit of $Q$ that
are $\varepsilon$-perturbations of $Df(f^i(Q))$ and satisfy the following
properties:
\begin{enumerate}
\item[{\bf{(A)}}]
the composition $B_t=A_{\ell-1,t}\circ \cdots \circ A_{0,t}$ is
hyperbolic for all $t\in [0,1]$,
\item[{\bf{(B)}}]
$A_{i,0}=Df(f^i(Q))$ for all $i=0,\dots \ell-1$, and
\item[{\bf{(C)}}]
the multipliers $\la_m$  and the exponents of $\chi_m$ of the
composition $B_1$ satisfy the conclusions in Proposition~~\ref{p.complex}:
$$
\chi_j=\chi_{j+1}, \qquad \chi_m\neq \chi_j \quad  \mbox{if $m\neq
j,j+1$}, \qquad \la_j, \la_{j+1}\not\in \RR.
$$
\end{enumerate}
\end{lemm}

We briefly introduce some formalism that we will use only in this section. 
Consider a set $\Sigma$ and a bijection $g\colon \Sigma \to \Sigma$.
Let $E$ be a vector bundle over the 
base $\Sigma$ such that its fibers $E_x$, $x\in\Sigma$,  are endowed with an Euclidean metric. 
A {\emph{linear cocycle}} on $E$ over $g$ is a map 
$\cA\colon E \to E$ that sends each fiber $E_x$ to a fiber $E_{g(x)}$ by a linear isomorphism $\cA_{x}$. 
The map $g$ is called the {\emph{base transformation}} of the cocycle $\cA$. 

The {\em distance} between two linear cocycles $\cA$ and $\cB$ above the same base transformation $g$ is 
$$
\dist(\cA,\cB)=\sup_{x\in \Sigma}\{\|\cA_{x}-\cB_{x}\|,\|{(\cA_{x}})^{-1}-{(\cB_{x})}^{-1}\|\}.
$$
A {\emph{path of cocycles}} defined on the bundle $E$ is a one-parameter family of cocycles $(\cA_t)_{t\in[0,1]}$ above the same base transformation $g$
such that the map $t\mapsto \cA_t$ is continuous for the metric above. 
The {\emph{radius}} of the path $(\cA_t)_{t\in[0,1]}$ is defined by
$$
\max_{t\in [0,1]}
\dist(\cA_0,\cA_t).
$$
Here we only deal with continuous cocycles (for the ambient topology of $E$) 
whose base transformations are diffeomorphisms or  restrictions of diffeomorphisms to invariant subsets of the ambient.

Finally, hyperbolicity and domination of cocycles are defined in the natural way, see for example
Definition~\ref{d.severalbundles}.

We will deduce Lemma~\ref{l.pcomplex} from the following result:
\begin{lemm}\label{l.pathpcomplex}
Consider a homoclinic class $H(P,f)$ and $j\in \NN$ satisfying the
hypothesis of Proposition~\ref{p.complex}.  Then there is an arbitrarily small path of
continuous cocycles $(\cA_t)_{t\in[0,1]}$ on $TM$
above the diffeomorphism $f$,
a point $\bar Q$ homoclinically related to $P$,
and a horseshoe $K$ containing $\bar Q$ 
such that:
\begin{itemize}
\item $\cA_0$ coincides with $Df$,
\item the cocycle $\cA_t$ restricted to $T_KM$ is hyperbolic, for all $t\in[0,1]$,
\item the cocycle $\cA_1$ restricted to $T_KM$ 
has a dominated splitting 
$$
T_K M=E\oplus E^{j,j+1}\oplus F
$$ 
such that $E$ has dimension $j-1$ and $E^{j,j+1}$ has dimension $2$,
\item 
the cocycle $\cA_1$ restricted
to the (periodic) orbit of $\bar Q$ by $f$
does not admit any dominated splitting over $E^{j,j+1}$.
\end{itemize}
\end{lemm}

Here a small path, means path of small radius.

\begin{proof}[Proof of Lemma~\ref{l.pathpcomplex}]

Observe first that arguing as in the previous propositions we
just get a periodic point $\bar Q$ homoclinically related to $P$
and a small path of hyperbolic cocycles $(\bar A_{i,t})_{t\in[0,1]}$,
$0\leq i<\pi(\bar Q)$, defined over the orbit of $\bar Q$ such that
the Lyapunov exponents  of the final composition $\bar B_1$
are real and
$\chi_j(\bar B_1)$ and $\chi_{j+1}(\bar
B_1)$ are equal, see
Proposition~\ref{p.gbobo}. Moreover, by Remark~\ref{r.bobo}, we
can assume that, for all $m\ne j,j+1$, the $m$-th exponent
$\chi_m(\bar B_1)$ is different from $\chi_j(\bar B_1)=\chi_{j+1}
(\bar B_1)$ for all $m\ne j,j+1$. 
Note that if the  multipliers $\lambda_j(\bar B_1)$ and $\lambda_{j+1}(\bar B_1)$ 
are equal then one can make 
them non-real and conjugate by an arbitrarily  small perturbation. 
However they might have opposite signs, which is why Lemma~\ref{l.pcomplex} is not obvious. 

We now go to the details of the proof of the lemma.
The path $\cA_t$ of cocycles is obtained as a concatenation of the following three paths.
First fix a transverse homoclinic point $X$ for $\bar Q$ and let
$$
\Lambda=\{f^n(\bar Q), 0\leq n <\pi(\bar Q)\} \cup \{f^n(X), n\in \ZZ\}.
$$
The compact invariant set $\Lambda$ is hyperbolic for the cocycle $Df$. 
\begin{itemize}
\item The first path $\cA_t^{[1]}$ ``linearizes" the dynamics  around  $\bar Q$. 
\item The second path $\cA_t^{[2]}$ is a path of cocycles on $TM$ that extends the path $(\bar A_{i,t})_{t\in[0,1]}$ of cocycles over the orbit of $\bar Q$ introduced above
in such a way that the set $\Lambda$ is a hyperbolic set for all $t$. 
\item The third path $\cA_t^{[3]}$  provides a  cocycle having the required dominated splitting over a horseshoe containing the set $\Lambda$.
\end{itemize}

For simplicity of notations, we will assume that $\bar Q$ is a fixed point for $f$ (the argument is identical in the general case). 
Thus we write  $(\bar A_{t})_{t\in[0,1]}$ instead of $(\bar A_{i,t})_{t\in[0,1]}$.
In what follows, the path of cocycles $(\bar A_{t})_{t\in[0,1]}$ becomes a path of matrices of $GL(d,\RR)$.

\medskip

\noindent \textbf{(I) The first path of cocycles $\cA_t^{[1]}$.} 
Fix a chart around the point $\bar Q$ so that for any $x$ in a neighborhood $V$ of the orbit of $\bar Q$, we can identify the derivative $Df$  (or any neighboring cocycle) at $x$ to a matrix of $GL(d,\RR)$. 

\begin{clai}\label{c.path1}
There is an arbitrarily small path $(\cA_t^{[1]})_{t\in [0,1]}$ of continuous cocycles  on $TM$ above $f$ 
and a neighborhood $W\subset V$ of $\bar Q$ such that 
\begin{itemize}
\item
$\cA_0^{[1]}=Df$,
\item by considering the restriction to the fiber of each point $x\in W$, the cocycle $\cA_1^{[1]}$ is identified to the derivative of $f$ at $\bar Q$,
\item the set $\Lambda$ is hyperbolic for all the cocycles $\cA_t^{[1]}$.
\end{itemize}
\end{clai}

\begin{proof}
By a unit partition, we build a cocycle $\cA_1^{[1]}$ satisfying the second item of the claim, for some small neighborhood $W$ of $\bar Q$. 
This cocycle can be chosen arbitrarily close to $Df$ (just take $W$ small enough). On each fiber of $TM$, 
take for the matrix of $\cA_t^{[1]}$  the $(1-t,t)$-barycenter of the matrices of $Df$ and $\cA_1^{[1]}$. 
Since the set $\Lambda$ is hyperbolic for $Df$, it will also be for all the cocycles $\cA_t^{[1]}$, provided we chose $\cA_1^{[1]}$ close enough to $Df$.
\end{proof}

\noindent \textbf{(II) The second path of cocycles $\cA_t^{[2]}$.} 
Fix a neighborhood $W$ of $\bar Q$ and a path  $(\cA_t^{[1]})_{t\in [0,1]}$ as in Claim~\ref{c.path1}.

\begin{clai}\label{c.path2}
There is a  path $(\cA_t^{[2]})_{t\in [0,1]}$ of continuous cocycles on $TM$ above $f$ such that: 
\begin{itemize}
\item
$\cA_0^{[2]}=\cA_1^{[1]}$,
\item its  radius is arbitrarily close to that of $\left(\bar A_t\right)_{t\in[0,1]}$, 
\item $\cA_{1}^{[2]}$ coincides with $\bar A_{1}$ over $\bar Q$,
\item for all $t\in [0,1]$, the set $\Lambda$ is hyperbolic for the cocycle $\cA_t^{[2]}$.
\end{itemize}
\end{clai}

\begin{proof}
For all $t\in [0,1]$, denote by $E^u_t$ and $E^s_t$ the stable and unstable directions of the hyperbolic point $\bar Q$
 for the cocycle $\bar A_t$. These directions vary continuously with $t$. 
Hence given any $\epsilon>0$ there exists a sequence $0=t_0<....<t_N=1$
 of times such that, for all $0\leq n < N$, there is a path of linear maps $\theta_{n,t}\in GL(d,\RR)$, with $\theta_{n,t_n}=Id$, and for all $t_n\leq t\leq t_{n+1}$:
\begin{itemize}
\item $\theta_{n,t}$ is $\epsilon$-close to identity,
\item $\theta_{n,t}(E^u_{t_n})=E^u_t$ and $\theta_{n,t}(E^s_{t_n})=E^s_t$.
\end{itemize}
Consider
$n_0\in \NN$  such that $f^{\pm n}(X)\in W$, for all $n\geq n_0$. 
First, we define the cocycle $\cA^{[2]}_{t}$ over the segment of orbit $\{f^n(X)\}_{n \ge 0}$  and 
for all $t\in[0,1]$. We  denote by $\cB_{n,t}$ the linear map corresponding to the cocycle $\cA^{[2]}_{t}$ over the point $f^n(X)$. For all $t_n\leq t\leq t_{n+1}$, define $\cB_{n,t}$ as follows:
\begin{itemize}
\item$\cB_{k,t}$ coincides with $\cA_{1}^{[1]}$ at $f^n(X)$, if $0\leq k < n_0$,
\item$\cB_{n_0+k,t}=\bar A_{t_k}\circ \theta_{k,t_{k+1}}$, if $k<n$,
\item$\cB_{n_0+n,t}=\bar A_t\circ\theta_{n,t}$,
\item$\cB_{n_0+k,t}=\bar A_t$, if $k>n$.
\end{itemize}
Recall that the set $\Lambda$ is hyperbolic for $\cA^{[1]}_1$.  Let $E^s_X$ and $E^u_X$ be the stable and unstable directions at $X$ for the cocycle $\cA^{[1]}_1$. 
By construction, for all $k\geq n$, the composition $\cB_{k,t}\circ ...\circ \cB_{0,t}$ maps $E^s_X$ (resp. $E^u_X$) into a 
direction corresponding to the stable (resp. unstable) direction of  $\bar A_t$. 

We define $\cB_{n,t}$ symmetrically for the backward orbit $\{f^n(X)\}_{n \le 0}$ of $X$.

Let $\cA_{t,\Lambda}^{[2]}$ be the linear cocycle on $T_{\Lambda}M$ given by the linear 
maps $\cB_{n,t}$ over the orbit of $X$ and by the matrix $\bar A_t$ over the point $\bar Q$. 
Then the bundles $E^s_X$ and $E^u_X$ are uniformly contracted and uniformly expanded, 
respectively, by positive iterations of $\cA_{t,\Lambda}^{[2]}$, and conversely by negative iterations. 
Hence, the orbits of the bundles $E^s_X$ and $E^u_X$ provide a hyperbolic splitting for $\cA_{t, \Lambda}^{[2]}$ over $\Lambda$: the cocycle $\cA_{t, \Lambda}^{[2]}$ is hyperbolic.

Besides, by construction, the family $(\cA_{t,\Lambda}^{[2]})_{t\in[0,1]}$ is a path
of continuous cocycles starting at the restriction of $\cA_1^{[1]}$ to the set $\Lambda$. 
The radius of this path can be found close to the radius of $\left(\bar A_t\right)_{t\in[0,1]}$: just take $\epsilon>0$ small enough. 
Now, all we need to do is to extend the path $\cA_{t,\Lambda}^{[2]}$ of cocycles above the restriction of $f$ to $\Lambda$ to 
a small path $(\cA_t^{[2]})_{t\in [0,1]}$ of continuous cocycles above $f$ starting at $\cA_1^{[1]}$.

Note that, for all $n>n_0+N$, the matrix of $\cA_{t,\Lambda}^{[2]}$ is $\bar A_t$
at the iterate $f^{\pm n}(X)$.  So is it also at $\bar Q$. Fix a small neighborhood $U_{\bar Q}\subset M$ of $\bar Q$ and these iterates. Fix a small neighborhood $U_n$ for each other iterate $f^n(X)$. Do this such that we have a disjoint union 
$$
U=U_{\bar Q}\cup \bigcup_{-n_0-N}^{n_0+N} U_{n}.
$$ 

Let $1=\phi+\psi$ be a unit partition on $M$ such that $\phi=1$ on $\Lambda$ and $\phi=0$ outside of the set $U$. Let $\cA_t^{[2]}$ be the cocycle above $f$ whose matrix on the fiber $T_xM$ is the $(\phi(x),\psi(x))$-barycenter of the 2 following two matrices:
\begin{itemize}
\item the matrix of $\cA_1^{[1]}$ at $x$, 
\item  $\begin{cases}
 \mbox{the matrix of $\cA_{t,\Lambda}^{[2]}$ at $f^n(X)$, if $x\in U_n$},\\
 \mbox{the matrix $\bar A_t$ of $\cA_{t,\Lambda}^{[2]}$ at $\bar Q$, if $x\in U_{\bar Q}$}. 
 \end{cases}$
 \end{itemize}
The cocycle $\cA_t^{[2]}$ does restrict to $\Lambda$ as $\cA_{t,\Lambda}^{[2]}$. Choosing the neighborhood $U$ of $\Lambda$ small enough, one finds the radius of $(\cA_t^{[2]})_{t\in [0,1]}$ as close as wished to the radius of 
$(\cA_{t,\Lambda}^{[2]})_{t\in [0,1]}$, hence as close as wished to the radius of $\left(\bar A_t\right)_{t\in[0,1]}$.
This ends the proof of the claim.
\end{proof}

\medskip

\noindent \textbf{(III) The third path of cocycles $\cA_t^{[3]}$.}  
We fix paths $\cA_t^{[1]}$ and $\cA_t^{[2]}$ as in Claims~\ref{c.path1} and~\ref{c.path2}.

\begin{clai}\label{c.path3}
There is an arbitrarily small path of cocycles $(\cA_t^{[3]})_{t\in [0,1]}$ defined on $TM$ above the diffeomorphism $f$ such that: 
\begin{itemize}
\item
$\cA_{0}^{[3]}=\cA_{1}^{[2]}$, 
\item $\cA_{1}^{[3]}$ coincides with $\bar A_{1}$ at $\bar Q$,
\item $\cA_{1}^{[3]}$ admits, over the set $\Lambda$, a dominated splitting of the form 
$$
T_{\Lambda}M=E\oplus E^{j,j+1}\oplus F
$$ 
such that $E$ has dimension $j-1$, and $E^{j,j+1}$ has dimension $2$,
\item for all $t\in[0,1]$, the cocycle $\cA_t^{[3]}$ is hyperbolic over the set $\Lambda$.
\end{itemize}
\end{clai}

\begin{proof}
Since $\cA_{1}^{[2]}$ is equal to $\bar A_1$ at $\bar Q$, recalling the properties of the exponents of $\bar A_1$, 
we have that there is a dominated splitting $T_{\bar Q}M=E\oplus E^{j,j+1}\oplus F$ with the required dimensions and such that $E^{j,j+1}$ 
is either uniformly contracted or uniformly expanded by $\cA_{1}^{[2]}$. 
We need to extend this splitting to the whole orbit of $X$.

Observe that
there are  $(j-1)$ and $(j+1)$-dimensional spaces $E_X$ and $\tilde{E}_X$ at the point $X$ such that their positive iterations by $\cA_1^{[2]}$ 
converge to $E$ and $\tilde{E}=E\oplus E^{j,j+1}$, respectively. Symmetrically, there are $(j-1)$ and  $(j+1)$-co\-di\-men\-sio\-nal spaces $\tilde{F}_X$ 
and $F_X$ whose negative iterations by $\cA_1^{[2]}$  converge to $\tilde{F}=E^{j,j+1}\oplus F$ and $F$, respectively. 

One can perturb slightly $\cA_{1}^{[2]}$ at the point $X$ in order to make $\tilde{E}_X$ transverse to $F_X$ 
and $E_X$ transverse to $\tilde{F}_X$.  Then the iterates of $\tilde{E}_X$ 
and $F_X$ by the perturbed cocycle along the orbit of $X$ extend 
the dominated splitting $\tilde{E}\oplus F$ to the whole set $\Lambda$.  
Symmetrically, we get an extension of the dominated splitting $E\oplus \tilde{F}$ to the set $\Lambda$.  
Taking  $ E^{j,j+1}=\tilde E \cap \tilde F$ we get 
the dominated splitting $E\oplus E^{j,j+1}\oplus F$ over $\Lambda$ for that perturbed cocycle.

That perturbation of $\cA_{1}^{[2]}$ may be reached by an arbitrarily small path of cocycles 
$(\cA_{t}^{[3]})_{t\in [0,1]}$ on $TM$ such that $\cA_{1}^{[3]}$ coincides with $\bar A_{1}$ at $\bar Q$. 
In particular, it can be chosen so that $\cA_t^{[3]}$ is hyperbolic over  $\Lambda$ for all $t$.
\end{proof}

\noindent \textbf{End of the proof of Lemma~\ref{l.pathpcomplex}}.
Define the path $(\cA_t)_{t\in [0,1]}$ as the concatenation of the
paths $(\cA^{[1]}_t)_{t\in [0,1]}$,  $(\cA^{[2]}_t)_{t\in [0,1]}$, and  $(\cA^{[3]}_t)_{t\in [0,1]}$ given by Claims~\ref{c.path1}, \ref{c.path2}, and \ref{c.path3}. 
By construction, the path $(\cA_t)_{t\in [0,1]}$ can be found having radius arbitrarily close to the radius of $(\bar A_t)_{t\in [0,1]}$. 
Choosing $\bar Q$ conveniently, this last radius can  be taken arbitrarily small. 

Note that the diffeomorphism $f$
has horseshoes $K$  containing the set $\Lambda$ that are arbitrarily close to $\Lambda$ for the Hausdorff distance. Choosing 
the horseshoe
$K$ Hausdorff-close enough to $\Lambda$, we have the following:
\begin{itemize}
\item for all $t\in [0,1]$,
the cocycles $\cA_t$ are continuous on $TM$ and hyperbolic over $\Lambda$. Thus, by a compactness argument on $t\in[0,1]$, the cocycles $\cA_t$ are 
also hyperbolic over $K$ for all $t\in[0,1]$.
\item 
The dominated splitting $T_{\Lambda} M=E\oplus E^{j,j+1}\oplus F$ for $\cA_1=\cA_1^{[3]}$ extends to $K$, see \cite[Appendix B]{BDV:04}.
\end{itemize}
All the conclusions of Lemma~\ref{l.pathpcomplex} are then satisfied. This ends its proof.
\end{proof}

\begin{proof}[Proof of Lemma~\ref{l.pcomplex}]
Let $(\cA_t)_{t\in [0,1]}$, $\bar Q$, $K$, and $T_{K}M=E\oplus E^{j,j+1}\oplus F$ be as in Lemma~\ref{l.pathpcomplex}.
Consider a transverse homoclinic point $X$ of $\bar Q$, $X\in W^u_\loc
(\bar Q,f) \cap K$, and an iterate of it $f^r(X)\in  W^s_\loc
(\bar Q,f)\cap K$. These two points can be chosen arbitrarily
close to $\bar Q$.

\begin{figure}[htb]

\psfrag{Q}{$\bar Q$}
\psfrag{first}{first loop}
\psfrag{second}{second loop}
\psfrag{QQ}{$Q_n$}

   \includegraphics[width=5.5cm]{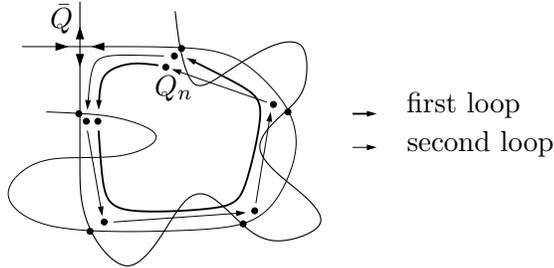}\hspace{1cm}
    \caption{Two-loops  orbits  $Q_n$}
    \label{f.orbit}
\end{figure}

We next consider periodic points $Q_n$ passing close to $X$ and having orbits with ``two loops".
For every large $n$ there is a
periodic point $Q_n\in K$ of period $2\,n+2+2\,r$ as follows (see
Figure~\ref{f.orbit}): Let $Q_n=Q_n^0$ and $Q_n^i=f^i(Q_n)$, where
\begin{equation}
\label{e.Qn}
\begin{split}
& \mbox{$\bullet\quad Q_n^0$ is close to $f^r(X)$ and $Q_n^0,\dots
,Q_n^n$ are
close to $\bar Q$,}\\
& \mbox{$\bullet\quad Q^{n+i}_n$ is close to $f^i(X)$ for all
$i=0,\dots,r$,}\\
& \mbox{$\bullet\quad Q_n^{n+r},\dots ,Q_n^{n+r+n+2}$ are close to
$\bar Q$,
and}\\
& \mbox{$\bullet\quad Q^{2\,n+r+2+i}_n$ is close to $f^i(X)$ for
all $i=0,\dots,r$.}
\end{split}
\end{equation}

\begin{clai}\label{c.orientation}
For $n$ large enough, the linear cocycle $\cA_1$ preserves the orientation of the central bundle $E^{j,j+1}$ at the periodic orbit of $Q_n$.
\end{clai}

\begin{proof} Let $\cA^c_1$ be the restriction of $\cA_1$ to the central bundle $E^{j,j+1}$.
Since the base $K$ of the $2$-dimensional bundle $E^{j,j+1}$ is a Cantor set, 
there is a continuous identification between $E^{j,j+1}$ and  $K\times \RR^2$. 
Thus, for any $x\in K$, the restriction $\cA^c_{1,x}$ of $\cA^c_1$ to the fiber $T_xM$ identifies to a $2\times 2$ matrix.   

By continuity, if the distance between a pair of points $x,y\in K$ is less than some $\eta>0$, then the determinants of the matrices 
$\cA^c_{1,x}$ and $\cA^c_{1,y}$ have the same sign. One then easily checks that for $n$ great enough (when ``close" in (\ref{e.Qn}) means distance less than $\eta/2$), the composition  of the matrices $\cA^c_{1,x}$ along the 
(finite) entire orbit of $Q_n$ has positive determinant.
\end{proof}

If the multipliers $\lambda_j$ and $\lambda_{j+1}$ of the first return map of $\cA_1$ at some $Q_n$ are complex, 
then all the conclusions of Lemma~\ref{l.pcomplex} are satisfied by $Q=Q_n$ and the restriction $(A_{i,t})_{t\in[0,1]}$ of the path $\cA_t$ to the orbit of $Q$.

Otherwise, by Claim~\ref{c.orientation},
these multipliers are real and have the same sign. 
Recall that the linear cocyle $\cA^c_1$ admits no dominated splitting at the point $\bar Q$. 
Since the orbits of $Q_n$ accumulate on $\bar Q$, then  with increasing $n$ the strength of domination (if any)
of the splitting of the bundle $E^{i,j}$ 
along the orbit of $Q_n$
 for the cocycle $\cA_1$ 
will decrease.
We can now apply \cite[Proposition 3.1]{BDP:03}. This result claims that, for $n$ great enough,  the
cocycle $\cA_1$ can be perturbed along the
 two-dimensional bundle $E_{j,j+1}$ and along the orbit of $Q_n$ to get a pair of non-real and
 conjugate eigenvalues. 
 
For $n$ great enough, that perturbation can be reached through a small path $(\cB_{t,n})_{t\in [0,1]}$ of cocycles over the orbit of $Q_n$. 
If the perturbation is small enough then, for all $t$,  the hyperbolicity and the domination of the splitting $E\oplus E^{i,j}\oplus F$ of $\cA_1$ over
the horseshoe
$K$ are preserved. Thus the conclusions of Lemma~\ref{l.pcomplex} are all satisfied for $Q=Q_n$ and the cocycle $(A_{i,t})_{t\in[0,1]}$ defined as the concatenation of
\begin{itemize}
\item the restriction of the path $\cA_t$ to the orbit of $Q=Q_n$, and
\item the path $\cB_{t,n}$. 
\end{itemize}
This concludes the proof of the Lemma~\ref{l.pcomplex}.
\end{proof}


\section{Formation of strong homoclinic connections}\label{s.formation}

We say that a saddle $P$ has a {\emph{strong homoclinic
intersection}} if there is a strong stable manifold of the orbit
of $P$ that intersects the unstable manifold of the orbit of $P$
or vice-versa. That is, let $i$ be the $s$-index of $P$, then either
$W^{ss}_k(P)\cap W^u(P)\ne \emptyset$ for some $k<i$ or
$W^{uu}_j(P)\cap W^s(P)\ne \emptyset$ for some $j< \dim (M)-i$
(recall the definitions of
$W^{ss}_k(P)$ and $W^{uu}_j(P)$ in Section~\ref{ss.basic}). In this
section, we see how the lack of domination of a homoclinic class
yields strong homoclinic intersections.

\begin{prop}\label{p.strongconnection}
For every $K>1$, $\varepsilon>0$ and $d\geq 2$, there exists a constant $k_0$
with the following property.

Consider $f\in\diff$, $\dim (M)=d$, and a hyperbolic periodic point $P$
of $s$-index $i$, $i\in \{2,\dots, d-1\}$, such that $H(P,f)$ is
non-trivial and has no $k_0$-dominated splitting of index $i-1$. Then
there is a periodic point $Q$
homoclinically related to $P$ and an $\varepsilon$-perturbation of $f$
along the orbit of $Q$ that is adapted to $H(P,f)$ and to property
$\fP_{ss}$ defined as follows
\begin{equation}\label{e.pss}
\fP_{ss}\eqdef
\begin{cases}
&\mbox{\textbf{(i)} $\chi_{i-1}(Q)<\chi_{i}(Q)$,}\\
&\mbox{\textbf{(ii)} $W^{ss}_{i-1}(Q)\cap W^u(Q)\ne\emptyset$}.
\end{cases}
\end{equation}
\end{prop}

\begin{proof}
By Proposition~\ref{p.complex} there is
a hyperbolic periodic point $Q$ that is homoclinically
related to $P$ and an $\frac \varepsilon 2$-perturbation $f'$ of $f$
along the orbit of $Q$ that is adapted to $H(P,f)$ and to property
$\fP_{i-1,i,\CC}$ (see equation \eqref{e.complex}).
This means that fixed small $\vro>0$, a neighborhood $V$ of the
orbit of $Q$, and compact sets $K^s\subset W^s_\vro(Q)$ and
$K^u\subset W^u_\vro(Q)$ disjoint from $V$, there is a diffeomorphism
$f'$ that is $\frac \varepsilon 2$-close to $f$ such that
\begin{enumerate}
\item\label{ipg1} $f'=f$ outside $V$ and along the $f$-orbit of
$Q$,
\item \label{ipg2} the points $P$ and $Q$ are homoclinically related for $f'$,
\item  \label{ipg3} $K^s\subset W^s_\vro(Q,f')$ and
$K^u \subset W^u_\vro(Q,f')$, and
\item \label{ipg4} the saddle $Q_{f'}=Q$ satisfies property $\fP_{i-1,i,\CC}$.
\end{enumerate}

By Remark~\ref{r.complex}, the period of $Q$ can be chosen arbitrarily large.
Hence Proposition~\ref{p.bobo} provides a small path of hyperbolic cocycles
joining the restriction of $Df'$ over the orbit
of $Q$ and  a cocycle with real multipliers. Applying
Lemma~\ref{l.gourmelon} to this cocycle and to $f'$ we get
an $\frac \pes 2$-perturbation $f''$ of $f'$, such that
$Q=Q_{f''}$ has a pair of real multipliers $\la_{i-1}(Q)$ and
$\la_{i}(Q)$ such that $|\la_{i-1}(Q)|=|\la_{i}(Q)|$ and
$|\la_{i}(Q)|\ne |\la_{j}(Q)|$ for all $j\ne i,i-1$,
and
 such that conditions (\ref{ipg1})--(\ref{ipg3}) also hold for
$f''$. Note that $f''$ is $\varepsilon$-close to $f$.

Consider now local coordinates around $Q$ such that
$$
W^s_\loc (Q,f'')= [-1,1]^i\times \{0^{d-i}\}\quad \mbox{and}\quad
W^u_\loc (Q,f'')= \{ 0^i\}\times [-1,1]^{d-i}.
$$
To conclude the proof of the proposition it is enough to get a
 diffeomorphism $g$ arbitrarily $C^1$-close to $f''$ and  a
small neighborhood $V_0\subset V$ of the orbit of $Q$ such that
\begin{enumerate}
\item[{\bf (a)}]
$g=f''$ outside $V_0$ and along the $g$-orbit of $Q_g=Q$,
\item[{\bf (b)}]
$W^s_\loc (Q,g)= [-1,1]^i\times \{0^{d-i}\}$ and $W^u_\loc (Q,g)=
\{ 0^i\}\times [-1,1]^{d-i}$, and
\item[{\bf (c)}]
$Q_g$ satisfies property $\fP_{ss}$.
\end{enumerate}

This will be done in several  steps. To simplify the presentation,
let us assume in the remainder steps of the proof that the period
of $Q$ is one.

\begin{clai}\label{cl.linear}
There is an arbitrarily $C^1$-small perturbation $g'$ of $f''$
satisfying (a) and (b) and such that the restriction of $g'$ to a
small neighborhood of $Q$ in $W^s_\loc(Q,g')$ is linear.
Moreover, one has that $D g'(Q)=D f''(Q)$.
\end{clai}

This claim allows us to define a two dimensional locally invariant
center-stable manifold $W^{cs}_\tau (Q,g')$ of $Q$ tangent to the
space corresponding to the $(i-1)$-th and $i$-th multipliers of
$Q$. Up to a linear change of coordinates, we have
$$
W^{cs}_\tau (Q,g')= \{ 0^{i-2}\} \times [-\tau,\tau]^2\times
\{0^{d-i}\} \text{ and }
Df''(Q)(x^s,x^u)=(A^s,A^u).
$$

\begin{proof}[Proof of Claim~\ref{cl.linear}]
 Using
the coordinates $(x^s,x^u)$ corresponding to the stable and
unstable bundles, in a neighborhood of $Q$, we write
$$
f''(x^s,x^u)=(f^s(x^s,x^u), f^u(x^s,x^u)).
$$
By the invariance of the local stable and unstable manifolds we
have that $f^u(x^s,0)=0^u$ and $f^s(0^s,x^u)=0^s$.

Next step is to linearize the restriction of $f^s$ to the local
stable manifold. Consider a  local perturbation $\tilde f^s$ of
$x^s \mapsto f^s(x^s,0^u)$  supported in a small neighborhood of
$0^s$ such that $\tilde f^s(x^s)=A^s(x^s)$ for small $x^s$.
Note that
$$
\tilde f^s(x^s)= f^s(x^s,0^u)+ h^s(x^s),
$$
where $h^s$ is $C^1$-close to the zero map and has support in a
small neighborhood of $0^s$.

Finally, we choose a bump-function $\psi(x^u)$ such that $\psi=1$
in a neighborhood of $0^u$, it is equal to $0$ outside another
small neighborhood of $0^u$, and it has small derivative. We
define $g'$ in a neighborhood of $(0^s,0^u)$ by
$$
g'(x^s,x^u)= \big( f^s(x^s,x^u) + h^s(x^s)\, \psi (x^u),
f^u(x^s,x^u) \big).
$$
By construction, the restriction of $g'$ to a small neighborhood
in the local stable manifold of $Q$ coincides with $\tilde
f^s=A^s$. Moreover, the local unstable manifold of $Q$ is also
preserved and $Dg'(Q)=Df''(Q)$. This completes the proof of the
claim.
\end{proof}

\begin{clai}\label{cl.homoclinic}
There is an arbitrarily small $C^1$-perturbation $g''$ of $g'$
satisfying Claim~\ref{cl.linear} (in particular, (a) and (b)) and such
that there is a transverse homoclinic point of $Q$ in $F_{i-1}$, where $F_{i-1}$ is
a $D g''(Q)$-invariant one-dimensional linear
 space corresponding to the Lyapunov exponent $\chi_{i-1}(Q)$.
\end{clai}

\begin{proof}
Note first that as the homoclinic class of $Q$ is non-trivial
there is a transverse homoclinic point $Y$ of $Q$ that belongs to
the local stable manifold of $Q$ where the dynamics is linear.
Next two steps are quite standard. First, by a perturbation we can
assume that $Y\not\in W^{ss}_{i-2}(Q)$.

Second, after replacing the point  $Y$ by some forward iterate of it and
after a new  perturbation, we can assume that $Y$ belongs to the
$D g'(Q)$-invariant (central) two-dimensional linear
 space $F$ corresponding to the Lyapunov exponents $\chi_{i-1}(Q)$ and
$\chi_{i}(Q)$. This follows noting that any stable non-zero vector of $Q$
that is not in the linear space corresponding to the
first $(i-2)$ Lyapunov exponents has normalized iterations
which approximate to  $F$.

There are two cases according to the restriction  of $g'$ to the
two dimensional space $F$.

\smallskip

\noindent {\emph{ Case 1: the restriction of $g'$ to   $F$ is a
homothety.}} In this case, the point $Y$ belongs to a
one-dimensional $Dg'(Q)$-invariant space and we are done.

\smallskip

\noindent {\emph{ Case 2: the restriction of $g'$ to   $F$ is
parabolic.}} In this case, the restriction of $Dg'(Q)$ to $F$ is
conjugate to a matrix of the form
$$ \left(
\begin{matrix} \la_{i} & 1
\\ 0 & \la_i
\end{matrix} \right), \quad  0<|\la_i|<1.
$$
Then the normalized iterations of any non-zero vector in $F$
accumulate to the unique one-dimensional invariant sub-space
$F_{i-1}$ of $Dg'(Q)$ in $F$. As above, after a new perturbation
we can assume that there is some iterate of $Y$ in $F_{i-1}$
ending the proof of the claim.
\end{proof}

To conclude the proof of the proposition it is now enough to make
the Lyapunov exponent $\chi_{i-1}(Q)$ smaller than $\chi_i(Q)$ so
that the space $F_{i-1}$ is now locally contained in the strong
stable manifold of $Q$ of dimension $i-1$. To perform this final
perturbation we argue as in Claim~\ref{cl.linear}.
\end{proof}


\section{Homoclinic tangencies yielding heterodimensional cycles}
\label{s.proofofmain}

In this section we prove Theorem~\ref{t.main} and its alternative
version in item (ii) of Remark~\ref{r.mainb}. For that we need the
following two propositions.

\begin{prop}\label{p.cycle}
Consider $f\in \diff$, $\dim (M)=d$, and a hyperbolic periodic point $P$ of $f$
of $s$-index $i\in \{2,\dots,d-1\}$ such that:
\begin{enumerate}
\item[(i)]
for any $C^1$-neighborhood $\cU$ of $f$ there exist a hyperbolic periodic
point $R$ homoclinically related to $P$ and perturbations of $f$
in $\cU$ along the orbit of $R$ that are adapted to $H(P,f)$ and
to property $\fP_{ss}$,
\item[(ii)]
for any $C^1$-neighborhood $\cU$ of $f$ and any $\delta>0$ there
exist a hyperbolic  periodic point $Q$ homoclinically related to $P$ and
perturbations of $f$ in $\cU$ along the orbit of $Q$ that are
adapted to $H(P,f)$ and to property $\fP_{i,\de}$,
\end{enumerate}

Then there exists a diffeomorphism $g\in \cU$ arbitrarily
$C^1$-close to $f$ having a heterodimensional cycle
associated to $P_g$ and a saddle $S_g$ of $s$-index $i-1$.
\end{prop}

Recall that properties $\fP_{ss}$ and $\fP_{i,\de}$, see
\eqref{e.pss} and \eqref{e.pid}, mean that the saddles $R$ and $Q$
satisfy
\[
\begin{split}
& \chi_{i-1}(R)<\chi_{i}(R) \quad \mbox{and} \quad
W^{ss}_{i-1}(R)\cap W^u(R)\ne\emptyset,\\
& \chi_i(Q)\in (-\de,0).
\end{split}
\]

%

%












%







%








\begin{prop}\label{p.excfinal}
Consider $f\in \diff$, $\dim (M)=d$, having a hyperbolic periodic point $P$ of
$s$-index $i\in \{2,\dots,d-1\}$ such that:
\begin{enumerate}
\item[{{(1)}}]
$H(P,f)$ is non trivial and has no dominated splitting of index
$i$,
\item[(2')]
there is a diffeomorphism $g$ arbitrarily $C^1$-close to $f$ with
a hyperbolic periodic point $R_g$ homoclinically related to $P_g$ satisfying
property $\fP_{ss}$,
 and
\item[(3')]
for every $\de>0$ there exists a hyperbolic periodic point $Q_\de$
homoclinically related to $P$ such that
$\chi_i(Q_\de)+\chi_{i+1}(Q_\de)\ge -\de$.
\end{enumerate}
Then, there exists a diffeomorphism $g$ arbitrarily $C^1$-close to
$f$ with a heterodimensional cycle associated to a
$P$ and to a saddle of $s$-index $i-1$.
\end{prop}

Note that item (1) in Proposition~\ref{p.excfinal}
corresponds exactly to the same item in Theorem~\ref{t.main},
items (2') and (3') are exactly items (2') and (3')
in Remark~\ref{r.mainb}.
Therefore Proposition~\ref{p.excfinal}
implies the conclusions in Remark~\ref{r.mainb}.

We postpone the proof of these propositions to
Sections~\ref{ss.pcycle} and \ref{ss.pexcfinal}.
Assuming these
propositions
we now prove
Theorem~\ref{t.main} and Corollary~\ref{c.main2}

\subsection{Proof of Theorem~\ref{t.main}}
\label{ss.proofoftmain}

Proposition~\ref{p.strongconnection}
and assumption (2) in the theorem imply that condition (i) in
Proposition~\ref{p.cycle} is satisfied.

Proposition~\ref{p.weak} and assumptions (1) and (3) in the
theorem imply that condition (ii) in Proposition~\ref{p.cycle}
is satisfied.

Proposition~\ref{p.cycle} now provides a diffeomorphism $g$ with a heterodimensional
cycle associated to $P_g$ and a saddle $S_g$ of $s$-index $i-1$.
By Lemma~\ref{l.bodiki}, we can assume that the diffeomorphism $g$ has
a pair of  transitive hyperbolic sets $L_g$ and $K_g$ having a robust heterodimensional cycle,
where  $L_g$ contains $P_g$ and $K_g$ contains a periodic point $R_g$ of stable index $i-1$.

We now explain how to improve the previous arguments to obtain
robust homoclinic tangencies.

Fix $\varepsilon>0$ and consider the integer $k_0$ associated to
$\varepsilon$ in Proposition \ref{p.complex}. Since $H(P,f)$ has
no dominated splittings of indices $i-1$ and $i$, there are $r>0$
and a neighborhood $\cU$ of $f$ such that for any $f'\in \cU$ and
any $f'$-invariant set $K$ having an $r$-neighborhood containing
$H(P,f)$ there is no $k_0$-dominated splitting over $K$.

We perform a first perturbation $g_0$ of $f$, $g_0\in \cU$, as
above, obtaining a robust heterodimensional cycle between two
transitive hyperbolic sets containing the saddles $P_{g_0}$ and
$R_{g_0}$. By~\cite{BC:04}, taking $g_0$ in a residual subset of
$\diff$, we can assume that $H(P,g_{0})$ and $H(R_{g_0},g_0)$
coincide. In particular, these  homoclinic classes are non-trivial
and their $r$-neighborhoods contain $H(P,g)$. Thus for every
diffeomorphism $h$ close to $g_0$, the homoclinic classes
$H(P_h,h)$ and $H(R_h,h)$ have no $k_0$-dominated splittings of
indices $i-1$ and $i$.

We now consider another small perturbation $g_1\in \cU$ of $g_0$
such that the saddles $P_{g_1}$ and $R_{g_1}$ have a
heterodimensional cycle.

Since the classes $H(P_{g_1},g_{1})$ and $H(R_{g_1},g_{1})$ have
no $k_0$-dominated splittings of indices $i-1$ and $i$,
Proposition \ref{p.complex} provides a pair of hyperbolic periodic
points $Q_{g_1}$ and $T_{g_1}$ homoclinically related to $P_{g_1}$
and $R_{g_1}$, respectively, and two ``independent" local
$\varepsilon$-perturbations $g_{Q}$ and $g_{T}$ of $g_1$ such that
\begin{itemize}
 \item
the   supports of $g_{Q}$ and $g_{T}$ are disjoint and contained
in arbitrarily small neighborhoods of the orbits of $Q_{g_1}$ and
$T_{g_1}$, respectively,
\item
these perturbations preserve the heterodimensional cycle
associated to $P_{g_1}$ and $T_{g_1}$,
\item
the $i$-th and $(i-1)$-th multipliers of $Q_{g_1}$  for $g_{Q}$
and of  $T_{g_1}$ for $g_{T}$ are non-real.
\end{itemize}
As the supports of the perturbations $g_{Q}$ and $g_{T}$ are
disjoint, combining these perturbations  one gets a diffeomorphism
$g_2$ such that $P_{g_2}$ and $T_{g_2}$ have a heterodimensional
cycle and the classes $H(P_{g_2},g_2)$ and $H(T_{g_2},g_2)$
robustly have no dominated splittings of indices $i$ and $i-1$,
respectively.

By Lemma~\ref{l.bodiki}, one can perform a last perturbation $g$
so that $P_g\in K_g$ and $T_g\in L_g$ where $K_g$ and $L_g$ are
transitive hyperbolic  sets having a robust heterodimensional
cycle. Finally, we choose $g$ in the residual subset of $\diff$
in \cite[Theorem 1]{BD:pre}, this choice implies that  the sets
$K_g$ and $L_g$ have robust homoclinic tangencies. \hfil \qed

\subsection{Proof of Corollary~\ref{c.main2}}
We first recall that there is a residual subset $\cR$ of $\diff$
such that every homoclinic class $H(P,f)$ of $f\in\cR$ that does
not have any dominated splitting is the Hausdorff limit of sinks
or sources, see \cite[Corollary 0.3]{BDP:03}. More precisely, if
there is a saddle $Q$ homoclinically related to $P$ whose Jacobian
is less (resp. greater) than one then the class $H(P,f)$ is the
Hausdorff limit of sinks (resp. sources), see the proof of
 \cite[Proposition 2.6]{BDP:03}. Thus to prove the corollary it
is enough to consider a saddle $P$ of $s$-index two whose
homoclinic class $H(P,f)$ does not have any dominated splitting
and such that every saddle $Q$ homoclinically related to $P$ has
Jacobian greater than one. By the previous comments, the class
$H(P,f)$ is limit of sources.

Observe that the assumption on the Jacobians implies that
$\chi_2(Q)+\chi_3(Q)>0$. Thus the homoclinic class satisfies all
hypotheses in Theorem~\ref{t.main}. Hence there is a perturbation
$g$ of $f$ with a robust heterodimensional cycle associated to a
hyperbolic set containing $Q_g$ and $P_g$. The corollary now
follows from standard genericity arguments. \hfil \qed

\subsection{Sectional dissipativiness. Corollary~\ref{c.excnodominated}}
Let $P$ be a hyperbolic saddle of a diffeomorphism $f$ such that:
\begin{itemize}
\item
for  every diffeomorphism $g$ that is $C^1$-close to $f$ there is
no heterodimensional cycle associated to $P_g$, and
\item
let $i$ the stable index of $P$, then the homoclinic class
$H(P,f)$ has no dominated splitting of index $i$.
\end{itemize}
 Under these hypotheses we consider a dominated splitting
with three bundles (see Definition~\ref{d.severalbundles})
$$
T_{H(P)}M=E_1\oplus E^c\oplus E_3
$$
 such that $\dim(E_1)<i<\dim(E_1\oplus E^c)$ and $E^c$ does not
 admit
 any dominated splitting.  Note that the bundles $E_1$ and $E_3$
 may
 be empty and that $\dim (E^c)\ge 2$.

We now see some properties of the homoclinic class $H(P,f)$ that
follow  from Theorem~\ref{t.main} and will imply the corollary.
There are the following cases:

\begin{itemize}
\item $\dim (E^c)=2$: Assume that $E^c$ is sectionally dissipative. Then, by Theorem~\ref{t.main} and
Remark~\ref{r.mainb}, for every diffeomorphism $g$ $C^1$-close to
$f$ and every saddle $R_g$ homoclinically related to $P_g$  the
unstable and strong stable manifolds of $R_g$ have empty
intersection. There is similar statement when $E^c$ is sectionally
dissipative for $f^{-1}$.
\item
$\dim(E^c)\ge 3$. Since the diffeomorphisms close to $f$ cannot
have heterodimensional cycles, Corollary~\ref{c.main}
implies that 
$$
{\textbf{(I)\quad}} i=\dim(E_1\oplus E^c)-1 \qquad \mbox{or}
\qquad {\textbf{(II)\quad}} i= \dim (E_1)+1.
$$
In case (I), by Theorem~\ref{t.main}, the bundle $E^c$ is
uniformly sectionally dissipative.  Moreover, by
Remark~\ref{r.mainb}, for every diffeomorphism $g$ $C^1$-close to
$f$ and every saddle $R_g$ homoclinically related to $P_g$ the
unstable and strong stable manifolds of $R_g$ have empty
intersection. There is similar statement for case (II) considering
$f^{-1}$.
\end{itemize}
The previous discussion implies Corollary~\ref{c.excnodominated}.
\hfill $\square$

\subsection{Proof of Proposition~\ref{p.cycle}}
\label{ss.pcycle}

We fix a small neighborhood $\cU$ of $f$ and small $\delta>0$.
Conditions (i) and (ii) in the proposition provide saddles $R$ and
$Q$ having different orbits and local perturbations $g_R$ and
$g_Q$ throughout these orbits as follows. Consider small
neighborhoods $V_R$ and $V_Q$ of the orbits of $R$ and $Q$ having
disjoint closures. Then there are perturbations $g_R$ and $g_Q$ of
$f$ in $\cU$ whose supports are contained in $V_R$ and $V_Q$ such
that $R$ satisfies $\fP_{ss}$ for $g_R$ and $Q$ satisfies
$\fP_{i,\de}$ for $g_Q$.

As the supports of these perturbations are disjoint, we can
consider a perturbation $g_0$ of $f$ which coincides with $g_R$ in
$V_R$, with $g_Q$ in $V_Q$, and with $f$
outside these neighborhoods. Note that if $\cU$ is small then the
diffeomorphism $g_0$ can be chosen arbitrarily close to $f$.
Moreover, since we are considering adapted perturbations, we have
that the saddles $R$ and $Q$ are all homoclinically
related to $P$ (recall the proof of Corollary~\ref{c.1}).

The proposition is an immediate consequence of the following two
claims. We observe that there are similar results  in
\cite{PPV:05} and \cite[section 2.5]{CP:prep}, so we just sketch
their proofs.

\begin{clai} \label{cl.weakandss}
There is a perturbation $g_1$ of $g_0$ having a hyperbolic
periodic point $S_{g_1}$ that is homoclinically related to
$P_{g_1}$ and that satisfies simultaneously properties $\fP_{ss}$
and $\fP_{i,\de}$.
\end{clai}

\begin{clai} \label{cl.getacycle}
The dynamical configuration in Claim~\ref{cl.weakandss} yields
diffeomorphisms $g$ having heterodimensional cycles associated to
a periodic orbit homoclinically related to $P_g$ and to a saddle
of index $i-1$. Moreover, if $\de>0$ is small and $g_1$ is close
enough to $f$ then $g\in \cU$.
\end{clai}

\begin{proof}[Sketch of the proof of Claim~\ref{cl.weakandss}]
The idea of the proof  of the claim is the following. First,
consider a strong homoclinic intersection $X$ of the orbit of $R$.
Then there are $N_1$ and $N_2>0$ such that
$$
X\in g^{-N_1}_0\big( W^{ss}_\loc(R, g_0) \big) \cap g^{N_2}_0
\big(W^{u}_\loc (R, g_0)).
$$

Observe also that, since $R$ and $Q$ are homoclinically related,
there is a locally maximal transitive hyperbolic set $L$ of $g_0$
containing $R$ and $Q$. Moreover, we can assume (and we do) that
$L$ is disjoint from the orbit of the point $X$.

We consider a ``generic" perturbation $g_0'$ of $g_0$ given by
Lemma~\ref{l.homocliniclyapunov} obtaining a periodic point $S_{g_0'}\in
L_{g_0'}$ which satisfies $\fP_{i,\de}$ and having iterates
arbitrarily close to $R_{g_0'}$. This implies that
$$
(g_0')^{-N_1} \big( W^{ss}_\loc({S_{g_0'} , g_0'})\big) \quad
\mbox{and} \quad (g_0')^{N_2} \big(W^{u}_\loc(S_{g_0'},
g_0')\big)
$$
have points that are close to $X$. Since $X$ is disjoint from the
orbit of $S_{g_0'}$ we can perform  a local perturbation $g_1$ of
$g_0'$ in a small in a neighborhood of $X$ having a strong
homoclinic intersection associated to $S_{g_1}$. This completes
the sketch of the proof of the claim.
\end{proof}

\begin{proof}[Sketch of the proof of Claim~\ref{cl.getacycle}]
By a small local perturbation $g_2$ of $g_1$ bifurcating the point
$S_{g_1}$ we get two points $\bar R_{g_2}$ and $\bar S_{g_2}$ of
indices $i-1$ and $i$ such that
\begin{itemize}
\item
 $\bar S_{g_2}$ is still
homoclinically related to $P_{g_2}$,
\item
the manifolds $W^u(\bar R_{g_2},g_2)$ and $W^s(\bar S_{g_2},g_2)$
have a transverse intersection point $Y$, and \item the $N_2$-th
iterate  of
 $W^u_\loc(\bar S_{g_2},g_2)$ and the $N_1$-th iterate by $g_2^{-1}$ of
 $W^s_\loc(\bar R_{g_2},g_2)$ have points that are close.
\end{itemize}

As
 above,  there is a small  local perturbation $g$ of $g_2$ such that
the intersection $W^u(\bar S_{g},g)\cap W^s(\bar R_{g},g)$ is
non-empty. The support of this perturbation is disjoint
 from the orbits of the saddles $\bar S_{g_2}$ and $\bar R_{g_2}$,
 the transverse intersection
point $Y$, and a pair of transverse heteroclinic points between
$\bar S_{g_2}$ and $P_{g_2}$. As a consequence, the diffeomorphism
$g$ has a heterodimensional cycle associated to $\bar S_g$ and
$\bar R_g$ and $\bar S_g$ is homoclinically related to $P_g$. This
completes the proof of the claim.
\end{proof}

This completes the proof of Proposition~\ref{p.cycle}. \hfill \qed

\subsection{Proof of Proposition~\ref{p.excfinal}}
\label{ss.pexcfinal} Consider any small $\varepsilon,\delta>0$.
The proof of this proposition follows exactly as the one of
Proposition~\ref{p.cycle}  after finding an
$\varepsilon$-perturbation $g_0$ of $f$ and two saddles $R$ and
$Q$ of $g_0$ that are homoclinically related to $P_{g_0}$ and
satisfy properties $\fP_{ss}$ and $\fP_{i,\de}$, respectively.

Let $k_0\geq 1$ be an integer associated to $\varepsilon$ given by
Proposition~\ref{p.weak}. Fix a point $Q=Q_\de$ as in item (3) in
the proposition. For an arbitrarily small perturbation $g'$  given
by item (2') consider the point $R_{g'}$ homoclinically related to
$P_{g'}$ and satisfying $\fP_{ss}$. Note that $Q_{g'}$ also
satisfies item (3). Moreover, the homoclinic class $H(P_{g'},g')$
does not have any dominated splitting of index $i$.

We now apply Proposition~\ref{p.weak} to get a perturbation $g_0$
of $g'$ supported on an arbitrarily small neighborhood of the
orbit of $Q_{g'}$ and such that property $\fP_{i,\de}$ holds for
$Q_{g_0}$ and $g_0$.
 Therefore all conditions in the  proposition are satisfied.
 \hfil \qed


\section{Viral classes}\label{s.viral}

In this section we prove  Theorem~\ref{t.bviral}. We begin with a
definition.

\begin{defi}[Property $\fV''$] \label{d.propertyV}
The chain recurrence class
 $C(P,f)$ of a saddle $P$ of a diffeomorphism $f\in \diff$, $\dim (M)=d$, satisfies
\emph{Property $\fV''$} if the following conditions hold:
\begin{enumerate}
\item\label{i.dviral1}
for every $j\in\{1,\dots,  d-1\}$ there exists a periodic point
$Q_{j}$ whose multipliers $\la_j(Q)$ and $\la_{j+1}(Q)$ are
non-real and whose Lyapunov exponents satisfy $\chi_{k}(Q)\ne
\chi_j(Q)$ for all $k\ne j,j+1$,
\item\label{i.dviral2}
let $i$ be the $s$-index of $P$, if $j$ is different from $i$ then
the points $P$ and $Q_{j}$ are homoclinically related,
\item\label{i.dviral3}
if $j=i$ then $Q_i$ has $s$-index $i+1$ or $i-1$ and there are two
hyperbolic transitive sets $L$ and $K$ containing $P$ and $Q_{i}$
and having a robust heterodimensional cycle, and
\item \label{i.dviral4}
there are saddles $Q^+$ and $Q^-$ homoclinically related to $P$
such that
\begin{equation}\label{e.propV}
\chi_1(Q^-)+\chi_2(Q^-)<0 \quad \mbox{and} \quad
\chi_{d-1}(Q^+)+\chi_d(Q^+)>0.
\end{equation}
\end{enumerate}
\end{defi}

Note that the points $Q_j$ in the definition belong to the chain
recurrence class $C(P,f)$. This is obvious for the saddles $Q_j$,
$j\ne i$, that are homoclinically related to $P$. For the saddle
$Q_i$ this follows from  the existence of the hyperbolic
transitive sets $L$ and $K$ containing $P$ and $Q_{i}$ and related
by a heterodimensional cycle.

 Note also that properties $\fV$,
$\fV'$, and $\fV''$ (recall Definitions~\ref{d.propertyS} and
\ref{d.propertySprime}) are  open by definition. The next two
lemmas imply that these three
properties are equivalent ``open and densely".

\begin{lemm} \label{l.VimpliesS}
Consider a saddle $P$ and its chain recurrence class $C(P,f)$. If
Property $\fV''$ holds for $C(P,f)$ then Property $\fV$ holds for
$C(P,f)$. Moreover, if the  dimension $d\geq 4$, then property
$\fV'$ also holds for $C(P,f)$.
\end{lemm}
\begin{proof}
Let $i$ be the $s$-index of $P$ and denote by $Q_{j}$ the saddles
in Property $\fV''$.
 Condition (\ref{i.dviral1}) and the fact that
 $Q_{j}$ belongs to
$C(P,f)$ robustly implies that there is a neighborhood $\cV_{j}$
of $f$ such that, for all $h\in \cV_j$, the class $C(P_h,h)$
cannot have a dominated splitting $E\oplus F$ of index $j$. Since
this holds for all $j=1,\dots,d-1$, the non-domination condition
follows for the class $C(P_h,h)$ for every diffeomorphism $h\in
\cV=\cap_{j=1}^{d-1} \cV_{j}$.

The fact that $C(P_h,h)$ contains a saddle of $s$-index different
from $i$ for all $h\in \cV$ follows from condition
(\ref{i.dviral3}) after recalling that $Q_{i,h}\in C(P_h,h)$ and
that its $s$-index is $i\pm 1$.
In dimension $d\geq 4$, either $P$ or $Q_i$ has s-index different from $1$ and $d-1$.
\end{proof}

\begin{lemm} \label{l.SimpliesV}
Consider a saddle $P_f$ and its chain recurrence class $C(P_f,f)$.
Let  $\cV$ be an neighborhood of $f$ such that Property $\fV$
holds for $C(P_g,g)$ for all $g\in \cV$. Then there is an open and
dense subset $\cW$ of $\,\cV$ such that $C(P_g,g)$ satisfies
$\fV''$ for all $g\in \cW$. In dimension $d\geq 4$, the same holds
when $\fV$ is replaced by $\fV'$.
\end{lemm}
\begin{proof}
Assume that $C(P_g,g)$ satisfies  property $\fV$ for all $g\in
\cV$. Let $i$ be the $s$-index of $P_f$.
Proposition~\ref{p.complex} implies that there is an open and
dense subset $\cW'$ of $\cV$ such that for all $j\ne i$ and all
$g\in \cW'$ there is a saddle $Q_{j,g}$ of $s$-index $i$
homoclinically related to $P_g$ whose $j$-th multipliers and
exponents satisfy condition (\ref{i.dviral1}). This implies items
(\ref{i.dviral1}) and (\ref{i.dviral2}) in Property~$\fV''$ for
$j\ne i$.

In what follows we use some properties of $C^1$-generic
diffeomorphisms. Given two hyperbolic saddles $P_f$ and $Q_f$ of a
generic diffeomorphism $f$ then $C(P_f,f)=H(P_f,f)$. Moreover, if
$Q\in C(P,f)$ then there is a neighborhood $\cU$ of $f$ such that
$Q_g\in C(P_g,g)$ for all $g\in \cU$, see \cite{BC:04}.
Furthermore, if $H(P,f)$ contains saddles of $s$-indices $i<j$
then it contains a saddle of $s$-index $k$ for all $k\in (i,j)\cap
\NN$,  see \cite{ABCDW:07}.

By the comments above, after a perturbation, we can assume that
the saddle $Q_g$ in Property~$\fV$  has $s$-index $i\pm 1$ for all
$g\in
 \cW'$. Let us assume, for instance, that this index is $i+1$.
  Note that $C(P_g,g)=C(Q_g,g)$ and that, by hypothesis, this
 class has no dominated splitting. Arguing as above, but now
 considering the saddle $Q_g$ of $s$-index $i+1$, we get saddles
 $Q_g'$ homoclinically
related to $Q_g$ whose multipliers and exponents satisfy condition
(\ref{i.dviral1}) for $j=i+1$. By construction, these saddles are
robustly in the same chain recurrence class of $Q_g$ and therefore
in $C(P_g,g)$.

By Corollary~\ref{c.bdk}, there exists two hyperbolic transitive
sets $L$ and $K$ containing $P_g$ and $Q'_g$ with a robust
heterodimensional cycle. Taking $Q_{i,g}=Q_g'$ we get condition
(\ref{i.dviral1}) for $j=i$ and condition (\ref{i.dviral3}).

Observe that condition (\ref{i.dviral4}) is trivial if the
$s$-index of $P_f$ is $i\ne 1,d-1$. Suppose that the index is $1$
(the case $d-1$ is analogous). In this case every  saddle $Q^+$
homoclinically related to $P_f$ satisfies
$\chi_{d-1}(Q^+)+\chi_d(Q^+)>0$. Note  that, after a perturbation
if necessary, we can assume that the homoclinic class of $P_f$
contains saddles $Q_f$ of stable index $2$.
After a new perturbation, one gets a diffeomorphism $h$ with a 
heterodimensional cycle
associated to  $Q_h$ 
and $P_h$.  
By the arguments in  \cite{ABCDW:07} (see Corollary 2)
the unfolding of these cycle provides diffeomorphisms $g$
with a saddle $Q_g^-$ homoclinically related to $P_g$ whose
Lyapunov
exponent $\chi_2 (Q_g^-)$ is arbitrarily close to $0^+$ while
$\chi_1(Q_g^-)$ is negative and uniformly away from $0$.
In particular, one has $\chi_{1}(Q_g)+\chi_2(Q_g)<0$.
This proves
that property $\fV''$ holds for $g$.
\medskip

When $d\geq 4$, let us now assume that $C(p_g,g)$ satisfies
property $\fV'$ for all $g\in \cV$. Corollary~\ref{c.main} implies
that there is a  dense and open subset of  $\cV$ consisting of
diffeomorphisms $g$ such that there exists a hyperbolic periodic
point $Q_g$ in $C(P_g,g)$ with s-index different from the s-index
of $P$. In particular Property $\fV$ holds and for a smaller dense
and open subset $\fV''$ holds.
\end{proof}
\bigskip

Theorem~\ref{t.bviral} is now a consequence of the two lemmas
above and the following proposition.

\begin{prop}
[Viral contamination]
 \label{p.viral}
Consider $f\in \diff$ and a saddle $P$ of $f$. Assume that the
chain recurrence class $C(P,f)$ of $P$ satisfies Property $\fV''$.
Then for every neighborhood $V$ of $H(P,f)$
 there exist a diffeomorphism $g$ arbitrarily
$C^1$-close to $f$ and a hyperbolic periodic point $Q_g$ of $g$
such that:
\begin{enumerate}
\item \label{i.pviral1}
the
  orbit of $Q_g$ is
arbitrarily close to $H(P_f,f)$ for the Hausdorff distance,
\item \label{i.pviral3bis}
there is an open neighborhood $U\subset V$ of the orbit of $Q_g$
such that $P \not\in U$ and either $f(\overline U)\subset U$ or
$f^{-1}(\overline U)\subset U$, and
\item \label{i.pviral2}
 $C(Q_g,g)$ satisfies Property $\fV''$.
\end{enumerate}
\end{prop}

Note that item (\ref{i.pviral3bis}) implies that $C(Q_g,g)$ is
disjoint from the chain-re\-cu\-rren\-ce class of $P_g$ (that
contains $H(P_g,g)$). Recall that property $\fV$ is robust. Thus
this proposition implies that Property $\fV''$ satisfies the
self-replication condition in Definition~ \ref{d.bviral}.

\subsection{Proof of
Proposition~\ref{p.viral}}\label{ss.proofofpviral}

We consider small $\pes>0$ and an upper bound $K$ of  the norms of
$Df$ and $Df^{-1}$. Let $k_0$ and $\ell_0$ be the constants associated
to $\pes$ and $K$ in Lemmas~\ref{l.gdcds} and \ref{l.bgv}. Let $i$
be the $s$-index of $P$. For clearness, we split the proof of the
proposition into six steps.

\smallskip

\noindent{\emph{Step I: Construction of the saddle $Q$.}} Consider
periodic points $Q^+$ and $Q^-$ as in equation \eqref{e.propV} in
Definition~\ref{d.propertyV}, i.e. $\chi_1(Q^-)+\chi_{2}(Q^-)<0$
and $\chi_{d-1}(Q^+)+\chi_{d}(Q^+)>0$.  Note that there exists a
locally maximal transitive hyperbolic set $\La_f$ such that
\begin{itemize}
\item $\La_f$ contains $P$, $Q^+$, and $Q^-$,
\item
$\La_f\subset H(P,f)$, and
\item
$\La_f$ is arbitrarily close to $H(P,f)$ for the Hausdorff metric.
\end{itemize}
In particular, the set $\La$ has no $k_0$-dominated splitting.

\begin{clai}\label{cl.hausdorff} There is  a
perturbation $g_0$ of $f$ such that the continuation $\La_{g_0}$
of $\La$ is the Hausdorff limit of the orbits of periodic points
$Q_{g_0}\in \La_{g_0}$ such that
\begin{equation}\label{e.clhausdorff}
\chi_1(Q_{g_0}) + \chi_2(Q_{g_0})<0 \quad \mbox{and} \quad
\chi_{d-1}(Q_{g_0}) + \chi_d(Q_{g_0})>0.
\end{equation}
Moreover, the set  $\La_{g_0}$ has no $k_0$-dominated splitting of
any index.
\end{clai}

\begin{proof}
If the $s$-index $i$ of $P$ belongs to $\{2,\dots,d-2\}$ then the
condition on the Lyapunov exponents holds for any saddle
homoclinically related to $P$. Thus it is enough to consider the
cases $i=1$ and $i=d$.

Let assume that $i=1$ (the case $i=d-1$ is similar). In this case,
$\chi_{d-1}(Q) + \chi_d(Q)>0$ for every saddle $Q$ that is
homoclinically related to $P$. Consider the saddle $Q^-\in \La$.
Taking a perturbation $g$ of $f$ in the residual set $\cG$ in
Lemma~\ref{l.homocliniclyapunov}, we can to ``spread" the property
$\chi_1(Q) +\chi_2(Q)<0$ over the hyperbolic set $\La_{g_0}$,
obtaining the point $Q_{g_0}$. This completes the first part of
the claim.

Since $g_0$ is close to $g$ and $\La_{g_0}$ is close to $\La$,
there is no $k_0$-dominated splitting over $\La_{g_0}$. This ends
the proof of the claim.
\end{proof}

By Lemma~\ref{l.dominatedclosure}, we can take the point $Q_{g_0}$
in Claim~\ref{cl.hausdorff} such that its orbit does not have any
$k_0$-dominated splitting, has period larger than $\ell_0$, and
its distance to the homoclinic class $H(P_{g_0},g_0)$ is
arbitrarily small. This completes the choice of the point
$Q=Q_{g_0}$.

\smallskip

\noindent{\emph{Step II: Separation of homoclinic classes.}} By
Lemma~\ref{l.bgv} there is an $\pes$-perturbation $g_1$ of $g_0$
supported on an arbitrarily small neighborhood of the orbit of
$Q_{g_0}$ such that the orbit of $Q_{g_1}$ is a sink or a source
for $g_1$. In what follows, let us assume that $Q_{g_1}$ is a
sink. Thus there is an open set $U\subset V$ containing the orbit
of $Q_{g_1}$ such that $g_1(\overline U)\subset U$ and $U$ is
disjoint from the homoclinic class of $P_{g_1}$. Note that these
properties hold for any diffeomorphism $g$ that is $C^0$-close to
$g_1$.
This implies item (\ref{i.pviral3bis}) in the proposition.

Recall that the choice of $Q$ and the neighborhood $U$ imply that,
for any perturbation $g$ of $g_1$, the homoclinic class $H(Q_g,g)$
is close to $H(P_f,f)$. This gives item (\ref{i.pviral1}) of the
proposition.

\smallskip

\noindent{\emph{Step III: Non-trivial homoclinic class of $Q$.}}
Note that after an $\pes$-per\-tur\-ba\-tion we can ``recover" the
original cocycle given by the derivative $Dg_0$ over the orbit of
$Q_{g_0}$, now defined over the orbit of $Q_{g_1}$. In particular,
there is no $k_0$-dominated splitting over the orbit of $Q_{g_1}$,
conditions in equation \eqref{e.clhausdorff} hold, and the saddle
$Q_{g_1}$ has $s$-index $i$. In what follows all perturbations $g$
we consider will preserve the cocycle over the  orbit of
$Q_{g_1}$. Hence the homoclinic class of  $Q_{g}$ will satisfy
item (\ref{i.dviral4}) in Property $\fV''$.

Finally, by Lemma~\ref{l.gdcds} and Remark~\ref{r.gdcds}, there is
an $\pes$-perturbation $g_2$ of $g_1$ supported on an arbitrarily
small neighborhood of the orbit of $Q_{g_1}$ such that the
homoclinic class of $Q_{g_2}$ is not-trivial.

\smallskip

\noindent{\emph{Step IV: No domination for the homoclinic class of
$Q$.}} Since there is no $k_0$-dominated splitting over the  orbit
of $Q_{g_1}$, by Corollary~\ref{c.1}, there is a
$\pes$-perturbation $g_3$ of $g_2$ such that for any $j\ne i$,
$j\in \{1,\dots, d-1\}$, there is a periodic point $Q_{j,g_3}$
homoclinically related to $Q_{g_3}$ that satisfies Property
$\fP_{j,j+1,\CC}$. In what follows, all perturbations that we will
perform  will preserve these properties. This implies that the
homoclinic class will satisfy items (\ref{i.dviral1}) and
(\ref{i.dviral2}) in the definition of Property~$\fV''$ for every
$j\ne i$.

Finally, for $j=i$, as the class $H(Q_{g_3},g_3)$ is not
$k_0$-dominated, using
 Lemma~\ref{l.gdcds} and
Remark~\ref{r.gdcds} we can generate a homoclinic tangency inside
the class after an $\pes$-perturbation $g_4$ of $g_3$. This
prevents the existence of a dominated splitting of index $i$ for
$g_4$.

Note that to complete the proof of the proposition it remains to
get items (\ref{i.dviral1}) for $j=i$ and (\ref{i.dviral3}) of
Property~$\fV''$.

\smallskip

\noindent{\emph{Step V: Generation of a robust heterodimensional
cycle.}} Recall that the homoclinic class $H(Q_{g_4},g_4)$ has no
any dominated splitting. There are three possibilities for the
$s$-index $i$ of $P$. If $i\in\{2,\dots,d-2\}$ we can apply
Corollary~\ref{c.main} to get $g_5$ close to $g_4$ with a robust
heterodimensional cycle associated to a hyperbolic set $L_{g_5}$
containing $P_{g_5}$ and a hyperbolic set $K_{g_5}$ of stable
index $i+1$ or $i-1$.

Assume now that $i=d-1$. Recall that $\chi_{d-1}(Q_{g_4}) +
\chi_d(Q_{g_4})>0$. Thus the hypotheses in Theorem~\ref{t.main}
are satisfied by $g_4$ and we get a diffeomorphism $g_5$ having a
robust heterodimensional cycle as before.

Finally, the case $i=1$ is analogous to the case $i=d-1$. Hence we
 obtain item (3) in Property~$\fV''$.

\smallskip

\noindent{\emph{Step VI: And finally Property~$\fV''$ holds.}}
Note that since the sets $L_{g_5}$ and $K_{g_5}$ have a robust
heterodimensional cycle, for all $g$ close to $g_5$ they are
contained in the same chain recurrence class. Thus by
\cite[Remarque 1.10]{BC:04} there is a residual subset $\cG'$ of
$\diff$ such that for every $f\in \cG'$, every periodic point of
$f$ is hyperbolic and its homoclinic and chain recurrence classes
coincide. In particular, for diffeomorphisms in $\cG'$ the
homoclinic classes of two periodic points either coincide or are
disjoint.

 Therefore for any $g_6\in \cG'$ close to
$g_5$ there is a periodic point $R_{g_6}\in K_{g_6}$ such that the
homoclinic classes $H(R_{g_6},g_6)$ and $H(Q_{g_6},g_6)$ coincide.
Hence the homoclinic class $H(R_{g_6},g_6)$ does not have any
$k_0$-dominated splitting of index $i$.

By Proposition~\ref{p.complex},
there is a saddle $Q_{i,g_6}$ homoclinically related to $R_{g_6}$
such that there is an $\pes$-perturbation of $g$ along the orbit
of $Q_{i,g_6}$ that is adapted to $H(R_{g_6},g_6)$ and to property
$\fP_{i,i+1,\CC}$. Since the perturbation is adapted, there is a
transitive hyperbolic set $K'_g$ containing $Q_{i,g}$ and $K_g$.
Thus the diffeomorphism $g$ has a robust heterodimensional cycle
associated to $L_g$ and $K'_g$. This ends the proof of the
proposition. \hfil \qed

\subsection{Proof of Corollary~\ref{c.bviral}} \label{ss.cviral}
Recall that the residual subset $\cG'$ of $\diff$ introduced in
Step VI consists of diffeomorphisms whose periodic points are all
hyperbolic. In particular, these diffeomorphisms have at most
countably many periodic points and hence countably many homoclinic
classes which are either disjoint or coincide.

By Lemma~\ref{l.SimpliesV}, there exists a dense open subset
$\cW\subset \cV$ such that $C(P_g,g)$ satisfies $\fV''$ for all
$g\in \cU$.

Recall that a filtrating neighborhood is an open set $U$ such that
$U=U_+ \cap U_-$ where $U_+$ and $U_-$ are open sets such that
$f(\overline{U_+}) \subset U_+$ and $f^{-1}(\overline{U_-})
\subset U_-$. Observe that there is
 filtrating neighborhood for the
chain recurrence class of  $Q_g$ separating this class and the
class of $P_g$. In particular, these two recurrence classes are
disjoint. Thus Theorem~\ref{t.bviral} allows to repeat this
process, generating new classes satisfying Property~$\fV''$.
Inductively, for each $n\in \NN$ we get an open and dense subset
$\cU_n$ of $\cU$ consisting of diffeomorphisms having (at least)
$n$ disjoint homoclinic classes. Therefore, the set
$$
\cG_\cU=\cG' \cap\ \bigcap_{n\in \NN} \cU_n
$$
is a residual subset of $\cU$ consisting of diffeomorphisms with
infinitely (countably) many homoclinic classes. This implies the
first part of the corollary.

To see that there are uncountably many chain recurrence classes
note that the first step of the construction provides two disjoint
filtrating open sets, the set $V_0=U$ containing the chain
recurrence class of $P_g=Q^0$ and the set $V_1$ containing the
chain recurrence class of $Q_g=Q^1$.

Repeating this process $n$ times, we can assume that for each map
$g\in \cG_\cU$  at each step we get $2^n$ open filtrating sets
$V_{i_1,\dots,i_n}$, $i_k=0,1$, that are pairwise disjoint and
nested (i.e. $V_{i_1,\dots,i_n}\subset V_{i_1,\dots,i_{n-1}}$),
and
 each set contains a  chain recurrence class with property
$\fV''$. Note that these classes are different and pairwise
disjoint.

Arguing inductively, we can repeat the construction of the first
step for every finite sequence $i_1,\dots,i_n$, getting a new pair
of filtrating neighborhoods $V_{i_1,\dots,i_{n},0}$ and
 $V_{i_1,\dots,i_n,1}$ contained in $ V_{i_1,\dots,i_n}$
and each of them containing a chain recurrence class satisfying
Property~$\fV''$.

Finally, for each infinite sequence $\iota=(i_k)$ consider the set
$$
K_\iota =\bigcap_{k=1}^\infty \overline{V_{i_1,\dots,i_k}}.
$$
By construction, each set $K_\iota$ contains some recurrent point
$X_\iota$ and given two different sequences $\iota$ and $\iota'$
the chain recurrence classes of $X_{\iota}$ and $X_{\iota'}$ are
different. Thus for $g\in \cG_U$ to each sequence $\iota$ we
associate a chain recurrent class $C(X_{\iota},g)$ and this map is
injective.

We have shown that every $g\in \cG_{U}$ has uncountably many chain
recurrence classes. Since, by the definition of $\cG'$, the
diffeomorphism $g$  has only countably many periodic points, there
are uncountably many aperiodic classes. This completes the proof
of the corollary. \hfil \qed

\subsection{Examples}\label{ss.examples}

We close this paper by providing examples of diffeomorphisms
satisfying viral properties that do not exhibit universal
dynamics.

\begin{prop}\label{p.nonempty}
Given any closed manifold $M$ of dimension $d\ge 3$ there is a
non-empty open set of diffeomorphisms having homoclinic classes
satisfying Property~$\fV$. Moreover, the open set can be chosen
such that the Jacobians of the diffeomorphisms are strictly less
than one over these homoclinic classes.
\end{prop}

The construction follows arguing exactly as in \cite[Appendix
6]{BD:02}. Just note that in this case we do not assume the
existence of a pair of points $P'$ and $Q'$ with Jacobians less
and larger than one as in \cite{BD:02}. A different approach is to
consider perturbations of systems having {\emph{heterodimensional
tangencies}} as in \cite{DNP:06}.



\begin{thebibliography}{100}





\bibitem{A:03}
F. Abdenur, \emph{Generic robustness of spectral decompositions,}
Ann. Sci. \'Ecole Norm. Sup., {\bf 36} (2003), 213--224.


\bibitem{ABC:}
F. Abdenur, Ch. Bonatti, and S. Crovisier, \emph{Nonuniform
hyperbolicity for $C^1$-generic diffeomorphisms,}
arXiv:0809.3309v1 and to appear in Israel Jour. of Math..


\bibitem{ABCDW:07}
F. Abdenur, Ch. Bonatti, S. Crovisier, L. J. D\'\i az, and L. Wen,
{\emph{Periodic points and homoclinic classes,}} Ergod. Th. and
Dynam. Syst.,  \textbf{27} (2007), 1--22.

\bibitem{AS:70}
R. Abraham and S. Smale, {\emph{Nongenericity of $\Omega
$-stability,}} Global Analysis (Proc. Sympos. Pure Math., Vol.
XIV, Berkeley, Calif., 1968), 5--8 Amer. Math. Soc., Providence,
R.I, (1970).






\bibitem{A:08}
M.~Asaoka, {\emph{Hyperbolic sets exhibiting $C\sp 1$-persistent
homoclinic tangency for higher dimensions,}} Proc. Amer. Math.
Soc. {\bf 136}  (2008),  677--686.

\bibitem{B:bible}
Ch. Bonatti, \emph{Towards a global view of dynamical systems, for
the $C^1$-topology,} pr\'e-publication Institut de Math\'ematiques
de Bourgogne (2010).



\bibitem{BB:pre} J. Bochi and
Ch. Bonatti, {\emph{Perturbation of the Lyapunov spectra of
periodic orbits,}} arXiv:1004.5029.



\bibitem{BC:04}
Ch. Bonatti and S. Crovisier, {\emph{R\'ecurrence et
g\'en\'ericit\'e,}} Inventiones Math., {\bf 158} (2004), 33--104.

\bibitem{BD:95}
Ch. Bonatti and L.J. D\'\i az, {\emph{Persistence of transitive
diffeomorphims,}} Ann. Math., {\bf 143} (1995), 367--396.

\bibitem{BD:99}
Ch. Bonatti and L.J. D\'\i az, {\emph{Connexions h\'et\'eroclines
et g\'en\'ericit\'e d'une infinit\'e de puits ou de sources,}}
Ann. Scient. \'Ec. Norm. Sup., {\bf 32}, 135-150, (1999).


\bibitem{BD:02}
Ch. Bonatti and L.J. D\'\i az, {\emph{On maximal transitive sets
of generic diffeomorphisms,}} Publ. Math. Inst. Hautes \'Etudes
Sci., {\bf 96} (2002), 171--197.

\bibitem{BD:08} Ch. Bonatti and L.J. D\'\i az,
{\emph{Robust heterodimensional cycles and $C^1$-generic
dynamics,}} Journal of the Inst. of Math. Jussieu, \textbf{7}
(2008),  469--525

\bibitem{BD:pre} Ch. Bonatti and L.J. D\'\i az, {\emph{Abundance of
$C^1$-robust homoclinic tangencies,}} to appear in Trans. A. M. S
and arXiv:0909.4062.


\bibitem{BDK:pre} Ch. Bonatti, L.J. D\'\i az, and S. Kiriki,
{\emph{Robust heterodimensional cycles and hyperbolic
continuations,}} in preparation.


\bibitem{BDP:03} Ch. Bonatti, L.J. D\'\i az, and E.R. Pujals,
{\emph{A $\cC^1$-generic dichotomy for diffeomorphisms: Weak forms
of hyperbolicity or infinitely many sinks or sources,}} Ann. of
Math., {\bf 158} (2003), 355--418.



\bibitem{BDV:04}
Ch. Bonatti, L.J. D\'\i az, and M. Viana, {\emph{Dynamics beyond
uniform hyperbolicity,}} Encyclopaedia of Mathematical Sciences
(Mathematical Physics), {\bf 102}, Springer Verlag, (2004).


\bibitem{BGV:06}
Ch. Bonatti, N. Gourmelon, and T. Vivier, \emph{Perturbations of
the derivative along periodic orbits,} Ergodic Th. and Dynam.
Syst., {\bf 26} (2006), 1307--1337.


\bibitem{Co:98} E. Colli, {\emph{Infinitely many coexisting strange
attractors,}} Ann. Inst. H. Poincar\'e Anal. Non Lin\'eaire, {\bf
15}  (1998), 539--579.

\bibitem{C:pre} S. Crovisier, {\emph{Birth of homoclinic intersections: a model for the central dynamics of partially hyperbolic systems,}}
arXiv:math/0605387 and to appear in Ann. Math..



\bibitem{CP:prep}
S. Crovisier and  E. R. Pujals, \emph{Essential hyperbolicity and
homoclinic bifurcations: a dichotomy phenomenon/mechanism for
diffeomorphisms}, in preparation.



\bibitem{DNP:06} L. J. D\'\i az, A. Nogueria, and E. R. Pujals,
 {\emph{Heterodimensional tangencies,}}   Nonlinearity, {\bf 19} (2006),
 2543--2566.

\bibitem{F:71} J. Franks, {\emph{Necessary conditions for stability of
diffeomorphisms,}} Trans. Amer. Math. Soc., {\bf 158} (1971),
301--308.


\bibitem{GW:03}
S. Gan and L. Wen, {\emph{ Heteroclinic cycles and homoclinic
closures for generic diffeomorphisms,}} J. Dynam. Differential
Equations, {\bf 15} (2003), 451--471.

\bibitem{G:10}
N. Gourmelon, {\emph{Generation of homoclinic tangencies by
$C^1$-perturbations,}} Discrete Contin. Dyn. Syst. {\bf 26}
(2010),  1--42.

\bibitem{G:pre}
N. Gourmelon, {\emph{A Franks' lemma that preserves invariant
manifolds,}}     arXiv:0912.1121v2.


\bibitem{M:78} R. Ma\~n\'e,
{\emph{Contributions to the stability conjecture,}} Topology, {\bf
17} (1978), 383--396.


\bibitem{M:pre}
C. G. Moreira, {\emph{There are no $C^1$-stable intersections of
regular Cantor sets,}} Pre-print IMPA 2008,
{\tt{http://www.preprint.impa.br/cgi-bin/MMMsearch.cgi}}


\bibitem{N:78}
S. Newhouse, {\emph{Diffeomorphisms with infinitely many sinks,}}
Topology, {\bf 13} (1974), 9--18.


\bibitem{N:79}
S. Newhouse, {\emph{The abundance of wild hyperbolic sets and
nonsmooth stable sets for diffeomorphisms,}} Publ. Math. I. H.E.S,
{\bf 50} (1979), 101--151.


\bibitem{N:04}
S. Newhouse, {\em New phenomena associated with homoclinic
tangencies,\/} Ergodic Theory Dynam. Systems, {\bf 24} (2004),
1725--1738.


\bibitem{PPV:05}
M. J. Pacifico, E. R. Pujals, and J. L. Vietez, {\emph{ Robustly
expansive homoclinic classes,}}  Ergod. Th. and Dynam. Syst., {\bf
25} (2005), 271--300.

\bibitem{P:00}
J. Palis, \emph{A global view of dynamics and a conjecture on the
denseness of finitude of attractors,} Ast\'erisque, {\bf 261}
(2000), 335--347.

\bibitem{PT:93}
J. Palis and F. Takens, {\emph{Hyperbolicity and sensitive chaotic
dynamics at homoclinic bifurcations. Fractal dimensions and
infinitely many attractors,}} Cambridge Studies in Advanced
Mathematics, {\bf 35}, Cambridge University Press, Cambridge,
(1993).


\bibitem{PV:94}
J. Palis and M. Viana, {\emph{High dimension diffeomorphisms
displaying infinitely many periodic attractors,}} Ann. of Math.,
{\bf 140} (1994), 207--250.




\bibitem{PS:00}
E. R. Pujals and M. Sambarino, \emph{Homoclinic tangencies and
hyperbolicity for surface diffeomorphisms,} Ann. of Math., {\bf
151} (2000), 961--1023.


\bibitem{R:95} N. Romero,
{\emph{Persistence of homoclinic tangencies in higher
dimensions,}} Ergodic Theory Dynam. Systems, {\bf 15} (1995),
735--757.


\bibitem{S:pre} K. Shinohara, {\emph{On the indices of periodic points in $C^1$-generic wild homoclinic classes in dimension three,}}
arXiv:1006.5571.

\bibitem{S:72}
C.P. Simon, {\emph{Instability in $\textrm{Diff}(T^3)$ and the
nongenericity
 of rational zeta function,}} Trans. A.M.S.,
 {\bf 174} (1972), 217--242.


\bibitem{W:04}
L. Wen, \emph{Generic diffeomorphisms away from homoclinic
tangencies and heterodimensional cycles,} Bull. Braz. Math. Soc.
(N.S.), {\bf 35} (2004), 419--452.


\end{thebibliography}
\end{document}